\documentclass[mathpazo]{csiam-tls}
\renewcommand\thisnumber{x}
\renewcommand\thisyear {2026}
\renewcommand\thismonth{xx}
\renewcommand\thisvolume{x}

\newcommand{\dd}{\mathrm{d}}
\usepackage{booktabs}            
\usepackage{enumitem}
\usepackage{diagbox}
\usepackage[normalem]{ulem}
\usepackage{tikz}
\usetikzlibrary{arrows.meta}
\usetikzlibrary{calc}

\begin{document}
\title{Curvature Effects on Boundary Accumulation of Self-Propelled Particles: A Fokker–Planck Approach}


\author[J. Fu et al.]{Jingyi Fu\affil{1},\ Jinqi Liu\affil{1},\ Benoit Perthame \affil{2},\ and Min Tang\affil{1}\corrauth}
\address{
\affilnum{1}\ School of Mathematics, Institute of Natural Sciences, Shanghai Jiao Tong University, 200240, Shanghai.\\
\affilnum{2}\ Sorbonne Université, CNRS, Université de Paris, Inria, Laboratoire Jacques-Louis Lions UMR7598, F-75005 Paris.\\
}
%
%
\emails{{\tt nbfufu@sjtu.edu.cn} (J.~Fu), {\tt lyuu\_shinnki@sjtu.edu.cn} (J.~Liu), {\tt Benoit.Perthame@sorbonne-universite.fr} (B.~Perthame), {\tt tangmin@sjtu.edu.cn} (M.~Tang)}

\begin{abstract}
It is well known that self-propelled particles tend to accumulate near boundaries. Although the influence of boundary geometry on this accumulation has been extensively studied using individual-based models, a population-level PDE description becomes desirable when a large number of particles are involved. In this work, we investigate boundary accumulation of self-propelled particles confined in two-dimensional circular and annular domains. We develop Fokker--Planck models that qualitatively reproduce the behavior observed in individual-based simulations. The main contribution is the formulation of appropriate evolution equations for particles accumulated at the boundary, incorporating transitions between the interior free-swimming state and the boundary-contact state. Starting from a general spatial two-dimensional model, we simplify the associated boundary conditions by assuming radial symmetry and then design an exponentially weighted difference scheme that preserves the total mass of the coupled system. Because transitions between different particle states induce singular behavior near the boundary, establishing discrete mass conservation is nontrivial. In particular, the velocity space must be divided into two subdomains, with the resulting discontinuity treated explicitly. Stochastic simulations and numerical experiments are performed both to validate the proposed models and to investigate the effects of geometric and dynamical parameters on boundary accumulation.
\end{abstract}


\maketitle

\section{Introduction}\label{ch:1}
Microorganisms are ubiquitous in natural environments and commonly inhabit spatially confined domains. Unlike passive Brownian particles, these microswimmers actively propel themselves through flagellar beating, ciliary motion, or other intrinsic mechanisms \cite{berg2004ecoli,elgeti2015physics,lauga2016bacterial}. Since their habitats are bounded by physical interfaces, interactions with confining boundaries are unavoidable and strongly influence transport, spatial organization, and collective dynamics \cite{bechinger2016active}. Consequently, understanding the effects of confinement is crucial in the study of active microorganisms.

 Experiments have consistently shown that active particles accumulate near boundaries \cite{cammann2021emergent,caprini2018active,elgeti2013wall,li2011accumulation,li2009accumulation,ostapenko2018curvature,rothschild1963nonrandom}. Several mechanisms have been proposed to explain such behavior, including hydrodynamic attraction \cite{berke2008hydrodynamic,giacche2010hydrodynamic}, steric interactions \cite{kantsler2013ciliary,li2009accumulation}, boundary-induced reorientation \cite{bianchi2017holographic} and short-range contact effects \cite{mousavi2020wall}. Boundary geometry, particularly local curvature, has been shown to substantially modify boundary-layer structure and particle distributions \cite{cammann2021emergent,duzgun2018active,nikola2016active,ostapenko2018curvature,wysocki2015giant}.


Physicists have extensively explored the effects of confinement through experiments, trajectory analyses and stochastic simulations of individual-based models. Ostapenko et al. experimentally tracked individual microalgae cells in circular and elliptical chambers and demonstrated that the probability of finding a cell near the boundary increases with wall curvature. The experimental results were well fitted by an individual-based Brownian-dynamics model, supporting the proposed mechanism of curvature-guided motility \cite{ostapenko2018curvature}. Building on this curvature-dependent navigation, Cammann et al. experimentally measured single-cell trajectories in confined geometries and combined these measurements with numerical simulations. They showed that spatial variations in boundary curvature generate nonequilibrium probability fluxes, whereas a circular domain with uniform curvature remains symmetric and exhibits no such fluxes \cite{cammann2021emergent}. These studies establish that boundary curvature can control not only where active particles accumulate, but also how they are transported within confined domains. For circular domains, Cruz et al. experimentally observed pronounced accumulation of active particles near the boundary and used a stochastic model to characterize their confined dynamics \cite{cruz2024active}. In annular confinement, Fazli and Naji used an individual-based model to investigate the steady-state distribution of active particles and found a pronounced difference in particle density between the inner and outer boundaries, demonstrating that the two boundaries contribute differently to particle accumulation \cite{zahra2021active}.

While particle-level approaches provide detailed information on individual trajectories and can reproduce observables such as particle density and flux, individual-based simulations are often computationally expensive and offer limited analytical insight into how boundary geometry regulates particle accumulation. In particular, it is generally difficult to derive explicit relations between local boundary curvature and the resulting near-boundary particle distribution from particle simulations alone.

The time evolution of the particle probability density function gives rise to continuum PDE models that offer a complementary population-level description. Kinetic PDE models have been proposed for stochastic velocity-jump processes and biased random walks \cite{alt1980biased,othmer1988models,stroock1974some}, and their corresponding diffusion and hydrodynamic limits have been extensively studied \cite{chalub2004kinetic,dolak2005kinetic,othmer2000diffusion,othmer2002diffusion}. Studying PDE models for the particle probability distribution not only reduces computational complexity but also enables a systematic investigation, using techniques from PDE analysis, of how confinement and boundary geometry influence the macroscopic dynamics.
 

To account for boundary effects, interactions with confining boundaries have traditionally been incorporated through boundary conditions in classical models for passive particles, such as gas molecules. More recently, boundary populations have been introduced as additional unknowns, leading to coupled bulk–boundary models in which near-boundary accumulation arises through exchange between particles in the interior and particles attached to the boundaries \cite{angelani2017confined,berlyand2020kinetic,ezhilan2015distribution,fu2023confined,fu2021fokker,das2018confined,malakar2018steady,nikola2016active,wagner2022steady}. However, most existing studies have focused on relatively simple geometries, such as planar walls. Extending these formulations to bounded domains with curved boundaries remains challenging, because the local geometry enters both particle transport and the coupling between bulk and boundary states. In this paper, we develop a continuum framework that can incorporate curvature effects for understanding particle distributions in more complex realistic confined environments. Fokker–Planck models for self-propelled particles in two-dimensional circular and annular domains are proposed, and used to investigate how boundary curvature shapes near-boundary accumulation.

The remainder of this paper is organized as follows. Sections \ref{ch:2} and \ref{ch:3} present the mathematical models for circular and annular domains, respectively, which account for particle exchange between the bulk and the boundary. Section \ref{ch:4} describes the numerical discretization, combining an upwind scheme in the radial direction with an exponentially weighted difference scheme in the tangential direction. Numerical results are presented in Section \ref{ch:5}, where we compare the individual-based models with the Fokker–Planck models. We conclude the paper in Section \ref{ch:6}.

\section{Models in Circular Domains}\label{ch:2}
\subsection{The Individual-Based Model}\label{ch:2.1}
Before introducing the governing equations of the individual-based model, we first state the assumptions underlying the model. Each particle moves with a constant speed \(V\), while its orientation \(\theta\) evolves according to a rotational diffusion process with strength \(D_\theta\). For a particle confined in a circular domain of radius \(R\), its direction of motion can be decomposed into radial and tangential components. When a particle reaches the boundary with an outward radial velocity, its radial motion is blocked. If its tangential velocity component is nonzero, the particle moves along the boundary, as illustrated in Figure \ref{fig:circulardemonstration}. Since the velocity decomposes naturally into radial and tangential components, polar coordinates \((r,\xi)\) are more appropriate for describing the particle dynamics.
\begin{figure}
    \centering
    \begin{tikzpicture}[>=Stealth]
\begin{scope}[shift={(0,0)}]
\draw[->] (-2.8,0) -- (2.8,0) node[right] {$x$};
\draw[->] (0,-2.8) -- (0,2.8) node[right] {$y$};
\draw (0,0) circle (2);
\node[below] at (2.35,0) {$R$};
\draw[dashed] (0,0) -- (-1.414,1.414);
\draw[dashed] (0,0) -- (1.414,-1.414);
\draw[dashed] (0,0) -- (1.414,1.414);
\draw[dashed] (-0.9473,1.3696) -- (-0.3,1.3696);
\draw[->] (-0.9473,1.3696) -- (-0.3473,1.9696);
\draw[->] (1.3696,-0.9473) -- (1.9696,-0.3473);
\draw[->] (0.814,0.814) -- (1.414,1.414);

\draw[->]
(100:1.8)
arc[start angle=100,end angle=48,radius=1.8];
\draw[->]
(-10:1.8)
arc[start angle=-10,end angle=42,radius=1.8];

\draw (-0.9473,1.3696) ++(0:0.3)
      arc[start angle=0,end angle=45,radius=0.3];
\node at (-0.9473+0.4,1.3696+0.2) {\small$\theta$};
\end{scope}
\end{tikzpicture}
    \caption{When rotational diffusion is neglected, a particle with orientation \(\theta=\pi/4\) first collides with the boundary and then moves along it. Its orientation remains unchanged, but its direction of motion changes due to confinement.}
    \label{fig:circulardemonstration}
\end{figure}
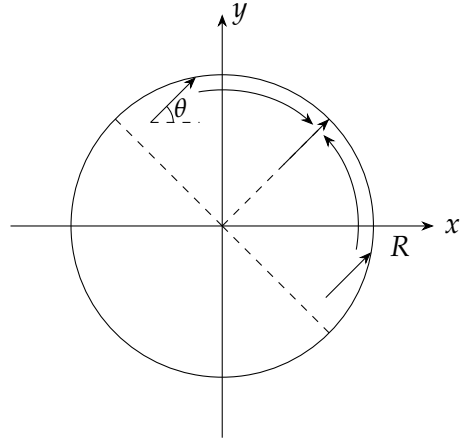

In polar coordinates, we formulate the individual-based model for particle dynamics in a two-dimensional circular domain. As in \cite{ostapenko2018curvature,fily2015dynamics}, the model neglects hydrodynamic interactions and collisions between particles. Let \((r_t,\xi_t)\) denote the position of the particle at time \(t\), and let \(\theta_t\) denote its orientation. The particle velocity can be decomposed into the  radial component $V_r = V\cos(\theta_t-\xi_t)$ and tangential component $V_\xi = V\sin(\theta_t-\xi_t)$ as illustrated in Fig. \ref{fig:illustration_of_movement} (a). The individual-based model is written as
\begin{equation}
\left\{
\begin{aligned}
\dd r_t &= 
\Biggl\{
\begin{aligned}
& V \cos(\theta_t-\xi_t)\,\dd t, && r_t < R\,\text{ or }\,r_t=R,\,\cos(\theta_t-\xi_t)<0,\\
& 0, && \text{otherwise,}
\end{aligned}
\Biggr.\\
\dd \xi_t &= \arctan\left(\frac{V\sin(\theta_t-\xi_t)\,\dd t}{V\cos(\theta_t-\xi_t)\,\dd t + r_t}\right),\\
\dd \theta_t &= \sqrt{2 D_\theta}\,\dd B_t.
\end{aligned}\label{eq:circularSDE}
\right.
\end{equation}

Each particle is represented by its position \((r_t,\xi_t)\in[0,R]\times[0,2\pi)\) and orientation \(\theta_t\in\mathbb{T}[0,2\pi)\).
It is important to note that in slab geometry with two parallel plates, the particle orientation coincides with its direction of motion. In a circular domain, however, the situation differs for self-propelled particles: the particle's orientation indicates the direction in which it wants to move, whereas its actual velocity indicates the direction in which it can move. These two directions are not necessarily the same at the boundary. Therefore, in \eqref{eq:circularSDE}, $\theta_t$ represents the direction in which the particle wants to move, while the actual velocity is given by $\frac{\dd r_t}{\dd t}\mathbf e_r+r_t\frac{\dd\xi_t}{\dd t}\mathbf e_\xi$.

The model in \eqref{eq:circularSDE} can be illustrated as follows.
 When \(r_t < R\), the particle lies inside the domain. Over a short time interval \(dt\), the change in \(r_t\) in the radial direction is \(V_r\,dt\), as shown in (\ref{eq:circularSDE}), while the change in \(\xi_t\) is given by the arctangent of  $V\sin(\theta_t-\xi_t)\dd t$ divided by $V\cos(\theta_t-\xi_t)\dd t +r_t$, as illustrated in Figure \ref{fig:illustration_of_movement} (b). We choose to use the arctangent form because it is more consistent with the geometric interpretation of the particle displacement over one time step. For $r_t>0$, it differs from $V\sin(\theta_t-\xi_t)/r_t\,\dd t$ only by a term of order $O(\dd t^2)$ and therefore yields the same continuous-time limit. When the particle is at the boundary \(r=R\) with an inward radial velocity component, i.e. the orientation of the particle is toward the interior of domain such that $\cos(\theta_t-\xi_t)< 0$, the coordinates \((r_t,\xi_t)\) evolve similarly. If the radial velocity component points outward ($\cos(\theta_t-\xi_t)\geq 0$), however, the radial motion is blocked, so \(r_t\) remains constant. In addition, \(\theta_t\) always evolves via a rotational diffusion process, where \(B_t\) denotes standard Brownian motion.

\begin{figure}[!tbh]
    \centering
\begin{tikzpicture}[>=Stealth]
\begin{scope}[shift={(0,0)}]
\draw[->] (-0.8,0) -- (4.8,0) node[right] {$x$};
\draw[->] (0,-0.8) -- (0,4.8) node[right] {$y$};
\node[below] at (4.15,0) {$R$};
\draw (4,0) arc[start angle=0,end angle=90,radius=4];

\coordinate (O) at (0,0);
\coordinate (P) at ({2*cos(30)},{2*sin(30)});

\draw[dashed] (O)--(P)
node[midway,above left,yshift=-3pt,xshift=2pt] {\small$r$};

\fill (P) circle (1.5pt);

\draw[-]
(0.7,0)
arc[start angle=0,end angle=30,radius=0.7];
\node at ({cos(15)},{sin(15)}) {\small$\xi$};

\coordinate (Vtip) at ($(P)+(75:2.0)$);
\draw[->] (P)--(Vtip)
node[near end,above,yshift=10pt,xshift=4pt] {\small$V$};
\coordinate (Pr) at ($(P)+(30:{2*cos(45)})$);
\coordinate (Pt) at ($(P)+(120:{2*sin(45)})$);

\draw[->] (P)--(Pr)
node[near end,right,yshift=2pt,xshift=6pt] {\small$V_r$};
\draw[->] (P)--(Pt)
node[near end,left,yshift=4pt,xshift=-2pt] {\small$V_\xi$};
\draw[dashed] (Pt)--(Vtip);
\draw[dashed] (Pr)--(Vtip);

\draw[dashed] (P)--($(P)+(0:{2*cos(60)})$);
\draw[-]
($(P)+(0.5,0)$)
arc[start angle=0,end angle=75,radius=0.5];

\node at ($(P)+(0.7,0.17)$) {\small$\theta$};
\end{scope}
\begin{scope}[shift={(8,0)}]
\draw[->] (-0.8,0) -- (4.8,0) node[right] {$x$};
\draw[->] (0,-0.8) -- (0,4.8) node[right] {$y$};
\draw (4,0) arc[start angle=0,end angle=90,radius=4];
\node[below] at (4.15,0) {$R$};

\coordinate (O) at (0,0);
\coordinate (P) at ({2*cos(30)},{2*sin(30)});
\coordinate (Vtip) at ($(P)+(75:2.0)$);
\draw[-] (O)--(P)
node[midway,below,yshift=5pt,xshift=12pt] {\small$r_t$};
\draw[->] (P)--(Vtip)
node[near end,above,yshift=10pt,xshift=4pt] {\small$V\mathrm{d}t$};
\draw[-] (O)--(Vtip)
node[midway,above,yshift=3pt,xshift=-5pt] {\small$r_{t+\dd t}$};
\draw
($(O)+(30:0.7)$)
arc[start angle=30,end angle=52.5,radius=0.7];
\node at ($(O)+(45:0.7)+(0.3,0.2)$) {\small$\mathrm{d}\xi_t$};
\coordinate (Pr) at ($(P)+(30:{2*cos(45)})$);
\draw[dashed] (P)--(Pr);
\draw[dashed] (Pr)--(Vtip);
\draw
($(P)+(30:0.5)$)
arc[start angle=30,end angle=75,radius=0.5];

\node at ($(P)+(45:0.7)+(0.2,0)$) {\small$\theta_t-\xi_t$};
\end{scope}
\end{tikzpicture}
\caption{Illustration of the notation. (a): decomposition of velocity, (b): changes of $r_t$ and $\xi_t$ during a short time $\dd t$.}
\label{fig:illustration_of_movement}
\end{figure}
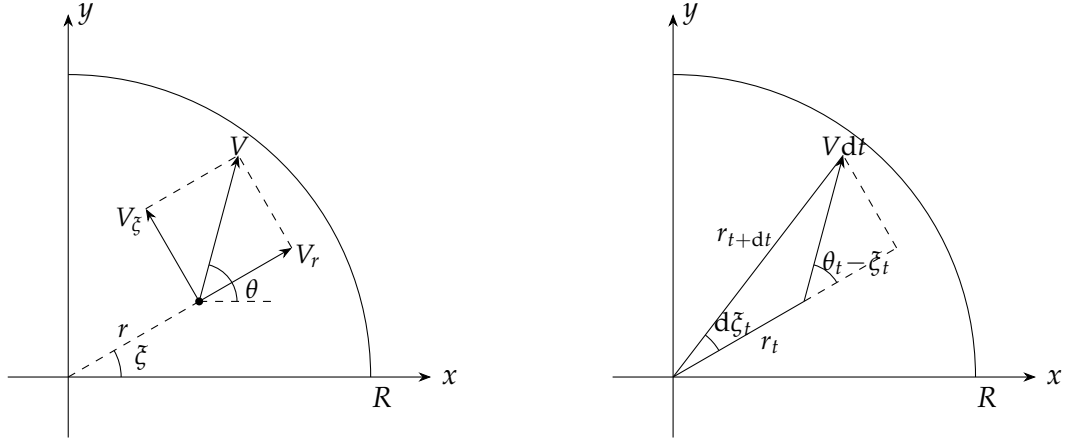

\subsection{The Fokker-Planck Model}\label{ch:2.2}
As in the simple case with two parallel plates \cite{fu2021fokker,fu2023confined}, since particles obey different motion laws in the interior of the domain and at the boundary, we introduce two probability density functions: one for freely moving particles in the interior and the other for particles confined at the boundary.
Let $p(t,r,\xi,\theta)$ denote the probability density of particles in the interior of the domain and at the boundary moving inward, at time $t$ and position $(r,\xi)$, with orientation $\theta$. $q(t,\xi,\theta)$ denotes the probability density function of particles at the boundary $\{r = R\}$ whose orientations point outward, i.e., $\cos(\theta-\xi)>0$. The Fokker-Planck system reads
\begin{equation}
\left\{
\label{eq:PDEsystem}
    \begin{alignedat}{2}
        &\frac{\partial(rp)}{\partial t}+V\cos(\theta-\xi)\frac{\partial(rp)}{\partial r}+\frac{V}{r}\frac{\partial\left[\sin(\theta-\xi)rp\right]}{\partial \xi}-D_\theta\frac{\partial^2(rp)}{\partial \theta^2}=0,&\quad &r\in[0,R],\\
        &\frac{\partial(Rq)}{\partial t}+\frac{V}{R}\frac{\partial\left[\sin(\theta-\xi)Rq\right]}{\partial \xi}-D_\theta\frac{\partial^2(Rq)}{\partial \theta^2}=V\cos(\theta-\xi)Rp(t,R,\xi,\theta),&\quad &\theta\in\left(\xi-\frac{\pi}{2},\xi+\frac{\pi}{2}\right).
    \end{alignedat}
    \right.
\end{equation}
 The equations are written in conservative form for the Jacobian-weighted densities $rp$ and $Rq$ in polar coordinates. The transport coefficients in the $r$ and $\xi$ variables are consistent with the transport velocities $V_r$ in the radial direction and $V_\xi$ in the tangential direction. Note that the velocity of the $\xi$ variable is the angular velocity and $r$ appears in the denominator. The second-order derivatives of $\theta$ takes into account the rotational diffusion and the right-hand side of the second equation in \eqref{eq:PDEsystem} indicates the inflow flux of particles transitioning from the free-swimming state to the boundary-contact state. When the particle is at the boundary and its orientation is pointing outside of the domain, the transportation in the radius direction is zero while the transportation in the tangential direction and the rotational diffusion are unchanged, as in the second equation in \eqref{eq:PDEsystem}. 




 The $\xi$ transportation can make the particle slide along the boundary as in Figure \ref{fig:circulardemonstration} and rotational diffusion can change the orientation direction of the particle. The particle can switch from $q(t,\xi,\theta)$ to $p(t,r,\xi,\theta)$ when its direction of motion changes from toward the boundary to away from the boundary.   
We can see from Figure \ref{fig:circulardemonstration} that when $r=R$, $\theta\in(\xi-\pi/2,\xi+\pi/2)$, the particles are moving towards and confined by the boundary, thus described by $q(t,\xi,\theta)$, while when $r=R$, $\theta\notin(\xi-\pi/2,\xi+\pi/2)$, particles are at free-swimming state, thus described by $p(t,r,\xi,\theta)$. The transitions from confined particles to free-moving particles are embedded in the boundary conditions of the PDE system (\ref{eq:PDEsystem}). The details are described below.

\begin{itemize}[leftmargin=*]
    \item \textbf{Boundary conditions for $p(t,r,\xi,\theta)$.}

    When the particle is at the boundary, while its velocity points into the domain, the particle will leave the boundary immediately, thus we have
\begin{equation}\label{eq:circular0dirichlet}
        p(t,R,\xi,\theta) = 0, \quad \cos(\theta-\xi)<0,
    \end{equation}
    On the other hand, when the particles pass through the origin of the domain, the polar angle should be shifted by $\pi$, continuity of the particle density at the origin is therefore expressed as 
    \begin{equation}
        p(t,0^+,\xi,\theta) = p(t,0^+,\xi+\pi,\theta),\quad\cos(\theta-\xi)>0.\label{eq:compatible}
    \end{equation}
$p(t,r,\xi,\theta)$ is naturally $2\pi$-periodic in the angular variables $\xi$ and $\theta$. More specifically, we have for $r\in(0,R)$,
    \begin{equation}
    \begin{aligned}
        &p(t,r,\xi,\theta) = p(t,r,\xi+2\pi,\theta),\\
        &p(t,r,\xi,\theta) = p(t,r,\xi,\theta+2\pi),
        \quad \frac{\partial}{\partial\theta}p(t,r,\xi,\theta) = \frac{\partial}{\partial\theta}p(t,r,\xi,\theta+2\pi).\label{eq:periodicxi_of_p}
    \end{aligned}
    \end{equation}
    Together with the $2\pi$-periodic boundary condition in the $\xi$ variable in \eqref{eq:periodicxi_of_p}, the boundary condition in (\ref{eq:compatible}) is valid for $\cos(\theta-\xi)<0$ as well. The condition \eqref{eq:compatible} is needed for the well-posedness of the PDE system \eqref{eq:PDEsystem}.

    \item \textbf{Boundary conditions for $q(t,\xi,\theta)$.}

    The particles switch from moving along the boundary to moving freely inside the domain once their polar angle $\xi$ and orientation $\theta$ satisfy $\cos(\theta-\xi)=0$ at $r = R$, hence we specify the following boundary conditions:
    \begin{equation}
        q\left(t,\xi,\xi-\frac{\pi}{2}\right) = q\left(t,\xi,\xi+\frac{\pi}{2}\right) = 0.
        \label{eq:Dirichlet_of_q}
    \end{equation}

    As before, $q$ is $2\pi$-periodic in $\xi$, hence we impose the following periodic boundary condition:
    \begin{equation}
        q(t,\xi,\theta) = q(t,\xi+2\pi,\theta).
    \end{equation}

    \item \textbf{Jump conditions due to transition.}
    
    It remains to characterize the particle behavior at $r=R$ and $\cos(\theta-\xi)=0$, where particles switch between the confined and freely moving  states, generating a flux exchange between $p$ and $q$. Integrating the PDE system (\ref{eq:PDEsystem}) at the point $(t,R,\xi,\xi+\pi/2)$, we get 
    \begin{equation}
    \begin{aligned}
    &\frac{\partial}{\partial t}\int_{(\xi+\pi/2)^-}^{(\xi+\pi/2)^+}rp(t,r,\xi,\theta)\dd\theta+\frac{V}{r}\frac{\partial}{\partial \xi}\int_{(\xi+\pi/2)^-}^{(\xi+\pi/2)^+}\sin(\theta-\xi)rp\dd\theta-D_\theta\frac{\partial}{\partial\theta}(rp)\bigg|_{(\xi+\pi/2)^-}^{(\xi+\pi/2)^+}\\
    =&-\left(\frac{\partial}{\partial t}\int_{(\xi+\pi/2)^-}^{(\xi+\pi/2)^+}Rq(t,\xi,\theta)\dd\theta+\frac{V}{R}\frac{\partial}{\partial \xi}\int_{(\xi+\pi/2)^-}^{(\xi+\pi/2)^+}\sin(\theta-\xi)Rq\dd\theta-D_\theta\frac{\partial}{\partial\theta}(Rq)\Bigg|_{(\xi+\pi/2)^-}^{(\xi+\pi/2)^+}\right)\delta_R(r).
    \end{aligned}    
    \end{equation}

    Assume that $p(t,r,\xi,\theta)\in L^\infty(\mathbb{R}^+\times[0,R]\times[0,2\pi)\times\mathbb{R})$ and $q(t,\xi,\theta)\in L^\infty(\mathbb{R}^+\times[0,2\pi]\times\mathbb{R})$.
    Using \eqref{eq:Dirichlet_of_q} we obtain
    \begin{equation}
    \frac{\partial}{\partial \theta}(rp)\bigg|_{\theta=(\xi+\pi/2)^-}^{\theta=(\xi+\pi/2)^+}=\frac{\partial}{\partial\theta}(Rq)\left(t,\xi,(\xi+\frac{\pi}{2})^-\right)\delta_R(r),\label{eq:jump+}
    \end{equation}
    which indicates a discontinuity in the $\theta$-derivative of $p$. Similarly, at $(t,R,\xi,\xi-\pi/2)$ we have
    \begin{equation}
    \frac{\partial}{\partial \theta}(rp)\bigg|_{\theta=(\xi-\pi/2)^-}^{\theta=\left(\xi-\pi/2\right)^+}=-\frac{\partial}{\partial\theta}(Rq)\left(t,\xi,\left(\xi-\frac{\pi}{2}\right)^+\right)\delta_R(r).\label{eq:jump-}
    \end{equation}

\end{itemize}

\paragraph{Mass Conservation.} As the particles cannot generate or escape, the PDE model (\ref{eq:PDEsystem}) has to satisfy the conservation of mass. This property can be described by the following theorem.

\begin{theorem}\label{th:massconservation}
Consider the PDE system (\ref{eq:PDEsystem}) with nonnegative bounded initial data:
\begin{equation}
    p(0,r,\xi,\theta) = p_0(r,\xi,\theta),\quad q(0,\xi,\theta) = q_0(\xi,\theta).
\end{equation}
The total mass 
\begin{equation}
    M(t) = \int_0^R\int_{0}^{2\pi}\int_{\xi-\pi/2}^{\xi+3\pi/2}rp(t,r,\xi,\theta)\dd\theta\dd\xi\dd r+\int_{0}^{2\pi}\int_{\xi-\pi/2}^{\xi+\pi/2}Rq(t,\xi,\theta)\dd\theta\dd\xi = I_1 + I_2,
\end{equation}
is a constant in time.
\end{theorem}
\begin{proof}
    We split the $\theta$ domain into a left part for $\theta\in(\xi-\pi/2,\xi+\pi/2)$ and a right part for $\theta\in(\xi+\pi/2,\xi+3\pi/2)$ and denote the mass in each part to be $M_\mathrm{left}$ and $M_\mathrm{right}$ such that
    $$
    I_1=\int_0^R\int_{0}^{2\pi}\int_{\xi-\pi/2}^{\xi+\pi/2}rp(t,r,\xi,\theta)\dd\theta\dd\xi\dd r+\int_0^R\int_{0}^{2\pi}\int_{\xi+\pi/2}^{\xi+3\pi/2}rp(t,r,\xi,\theta)\dd\theta\dd\xi\dd r:=M_\mathrm{left}+M_\mathrm{right}.$$
    Integrating the PDE system (\ref{eq:PDEsystem}) in the sense of almost everywhere and applying the Fubini Theorem as well as using boundary conditions for $p$, we have
\begin{equation}
\begin{aligned}
    &\frac{\dd M_{\mathrm{left}}}{\dd t}+\int_0^{2\pi}\int_{\xi-\pi/2}^{\xi+\pi/2}V\cos(\theta-\xi)Rp(t,R,\xi,\theta)\dd\theta\dd\xi
    -\int_0^R\frac{V}{r}\int_{-\pi/2}^{3\pi/2}rp\left(t,r,\theta+\frac{\pi}{2},\theta\right)\dd\theta\dd r\\-&\int_0^R\frac{V}{r}\int_{\pi/2}^{5\pi/2}rp\left(t,r,\theta-\frac{\pi}{2},\theta\right)\dd\theta\dd r-D_\theta\int_0^R\int_0^{2\pi}\frac{\partial}{\partial\theta}(rp)\bigg|_{\theta=(\xi-\pi/2)^+}^{\theta=(\xi+\pi/2)^-}\dd\xi\dd r=0.
\end{aligned}
\end{equation}
Similarly, for $M_\mathrm{right}$, we have

\begin{equation}
\begin{aligned}
    &\frac{\dd M_{\mathrm{right}}}{\dd t}
    +\int_0^R\frac{V}{r}\int_{\pi/2}^{5\pi/2}rp\left(t,r,\theta-\frac{\pi}{2},\theta\right)\dd\theta\dd r+\int_0^R\frac{V}{r}\int_{3\pi/2}^{7\pi/2}rp\left(t,r,\theta-\frac{3\pi}{2},\theta\right)\dd\theta\dd r\\
    -&D_\theta\int_0^R\int_0^{2\pi}\frac{\partial}{\partial\theta}(rp)\bigg|_{\theta=(\xi+\pi/2)^+}^{\theta=(\xi+3\pi/2)^-}\dd\xi\dd r=0.
\end{aligned}
\end{equation}
By adding the two equations and using periodic boundary conditions (\ref{eq:periodicxi_of_p}), we obtain
\begin{equation}
\begin{aligned}
    &\frac{\dd I_1}{\dd t}+\int_0^{2\pi}\int_{\xi-\pi/2}^{\xi+\pi/2}V\cos(\theta-\xi)Rp(t,R,\xi,\theta)\dd\theta\dd\xi+D_\theta\int_0^R\int_0^{2\pi}\frac{\partial(rp)}{\partial\theta}\bigg|_{\theta=(\xi+\pi/2)^-}^{\theta=(\xi+\pi/2)^+}\dd\xi\dd r\\
    +&D_\theta\int_0^R\int_0^{2\pi}\frac{\partial(rp)}{\partial\theta}\bigg|_{\theta=(\xi-\pi/2)^-}^{\theta=(\xi-\pi/2)^+}\dd\xi\dd r=0.\label{eq:dMdt}
\end{aligned}
\end{equation}
Integrating both sides of the equation for $q(t,\xi,\theta)$ and using Fubini Theorem and the Dirichlet boundary condition (\ref{eq:Dirichlet_of_q}), we have
\begin{equation}
\begin{aligned}
    &\frac{\dd I_2}{\dd t}-D_\theta\int_0^{2\pi}\frac{\partial(Rq)}{\partial\theta}\left(t,\xi,\left(\xi+\frac{\pi}{2}\right)^-\right)\dd\xi+D_\theta\int_0^{2\pi}\frac{\partial(Rq)}{\partial\theta}\left(t,\xi,\left(\xi-\frac{\pi}{2}\right)^+\right)\dd\xi\\
    =&\int_0^{2\pi}\int_{\xi-\pi/2}^{\xi+\pi/2}V\cos(\theta-\xi)Rp(t,R,\xi,\theta)\dd\theta\dd\xi.
\end{aligned}\label{eq:dM+dt}
\end{equation}
Combining (\ref{eq:jump+}), (\ref{eq:jump-}), (\ref{eq:dMdt}) and (\ref{eq:dM+dt}), we obtain directly that $\dd M/\dd t=0$, indicating that the total mass remains constant in time.
\end{proof}

    \paragraph{Radial-symmetric Model.}    
     The experiment in \cite{ostapenko2018curvature} exhibits a probability density distribution that is radius symmetric. If the solution to the system (\ref{eq:PDEsystem}) is radially symmetric, the standard way is to assume that the solution is independent of $\xi$. However, as we can see the boundary conditions depend heavily on $\theta-\xi$, the standard assumption no longer works. What really affects the dynamics of the system (\ref{eq:PDEsystem}) is the difference between the orientation $\theta$ and the polar angle $\xi$. Thus we introduce a new variable $\eta=\theta-\xi$ to reduce the dimensionality of the system. For the radial-symmetric model, assume $2\pi p(t,r,\xi,\theta)=\tilde p(t,r,\eta)$ and $2\pi q(t,\xi,\theta)=\tilde q(t,\eta)$, where the factor $2\pi$ results from integrating out the rotational variable $\xi$. Then $\tilde p$ and $\tilde q$ satisfy:
    \begin{equation}
    \left\{
    \begin{aligned}
     &\frac{\partial(r\tilde p)}{\partial t}+V\cos\eta\frac{\partial(r\tilde p)}{\partial r}-\frac{V}{r}\frac{\partial\left(\sin\eta\,r\tilde p\right)}{\partial \eta}-D_\theta\frac{\partial^2(r\tilde p)}{\partial \eta^2}=0,\quad &&r\in[0,R],\\
    &\frac{\partial(R\tilde q)}{\partial t}-\frac{V}{R}\frac{\partial\left(\sin\eta\,R\tilde q\right)}{\partial \eta}-D_\theta\frac{\partial^2(R\tilde q)}{\partial \eta^2}=V\cos\eta\,R\tilde p(t,R,\eta),\quad &&\eta\in\left(-\frac{\pi}{2},\frac{\pi}{2}\right).
    \end{aligned}\label{eq:2dsystem}
    \right.
    \end{equation}
    The corresponding boundary conditions become
    \begin{equation}
    \left\{
    \begin{aligned}
    &\tilde p(t,r,\eta) = \tilde p(t,r,\eta+2\pi),\quad \frac{\partial }{\partial\eta}\tilde p(t,r,\eta) = \frac{\partial }{\partial\eta}\tilde p(t,r,\eta+2\pi),\\
    &\tilde p(t,R,\eta) = 0,\quad \eta\in\left(\frac{\pi}{2},\frac{3\pi}{2}\right),\quad \tilde p(t,0^+, \eta) = \tilde p(t,0^+,\eta+\pi),\quad \eta\in\left(-\frac{\pi}{2},\frac{\pi}{2}\right),\\
    &\tilde q\left(t,-\frac{\pi}{2}\right)=\tilde q\left(t,\frac{\pi}{2}\right)=0,\\
    &\frac{\partial}{\partial\eta}(r\tilde p)\bigg|_{(\pi/2)^-}^{(\pi/2)^+}=\frac{\partial}{\partial\eta}(R\tilde q)\left(t,(\frac{\pi}{2})^-\right)\delta_R(r),\\
    &\frac{\partial}{\partial\eta}(r\tilde p)\bigg|_{(-\pi/2)^-}^{(-\pi/2)^+}=-\frac{\partial}{\partial\eta}(R\tilde q)\left(t,(-\frac{\pi}{2})^+\right)\delta_R(r).
    \end{aligned}
    \right.
    \end{equation}

\section{Models in Annular Domains}\label{ch:3}
\subsection{The Individual-Based Model}\label{ch:3.1}
For self-propelled particles, it has already been observed experimentally that the accumulation highly depends on the curvature of the boundary: when the curvature is concave, the particle is more likely to accumulate; when the curvature is convex, it is harder for the particle to accumulate \cite{cammann2021emergent,zahra2021active}. To understand this, we present in this section a model for particles confined in an annular domain. The annular domain has two boundaries of different curvatures, resulting in distinct particle behaviors along the inner and outer boundaries. 
Let the radii of the inner and outer boundaries be $R_\mathrm{in}$ and $R_\mathrm{out}$ respectively. We still use polar coordinates $(r_t,\xi_t)$ to represent the position of a particle at time $t$ and use $\theta_t$ to describe its orientation at the same time.

Fixing the particle orientation at $\theta_t=\pi/4$ and neglecting rotational diffusion, particles confined by the inner boundary will slide along the boundary in the directions plotted in Figure \ref{fig:annulardemonstration} (a). As $\xi_t$ changes with time, particles can only leave the boundary when $\xi_t=3\pi/4$ or $\xi_t=7\pi/4$, whereas the dynamics on the outer boundary are identical to those in the circular case, in which the particle cannot leave the boundary if the orientation does not change. 
\begin{figure}[!tbh]
    \centering
    \begin{tikzpicture}[>=Stealth]
    \begin{scope}[shift={(0,0)}]
    \draw[->] (-2.8,0) -- (2.8,0) node[right] {$x$};
    \draw[->] (0,-2.8) -- (0,2.8) node[right] {$y$};
    \draw (0,0) circle (2);
    \node[below] at (2.25,0) {$R_\mathrm{in}$};
    \draw[dashed] (0,0) -- (-1.414,1.414);
    \draw[dashed] (0,0) -- (1.414,-1.414);
    \draw[dashed] (0,0) -- (-1.414,-1.414);
    \draw[dashed] (-2.014,-2.014) -- (-1.328,-2.014);
    \draw[->] (-2.014,-2.014) -- (-1.414,-1.414);
    \draw[->] (-1.414,1.414) -- (-0.814,2.014);
    \draw[->] (1.414,-1.414) -- (2.014,-0.814);
    \draw[->]
    (-132:2.2)
    arc[start angle=-132,end angle=-45,radius=2.2];
    \draw[->]
    (-138:2.2)
    arc[start angle=-138,end angle=-225,radius=2.2];
    \draw (-2.014,-2.014) ++(0:0.3)
        arc[start angle=0,end angle=45,radius=0.3];
    \node at (-2.014+0.4,-1.85) {\small$\theta$};
    \end{scope}
    \begin{scope}[shift={(7,0)}]
    \draw[->] (-2.8,0) -- (2.8,0) node[right] {$x$};
    \draw[->] (0,-2.8) -- (0,2.8) node[right] {$y$};
    \draw (0,0) circle (2);
    \node[below] at (2.35,0) {$R_\mathrm{out}$};
    \draw[dashed] (0,0) -- (-1.414,1.414);
    \draw[dashed] (0,0) -- (1.414,-1.414);
    \draw[dashed] (0,0) -- (1.414,1.414);
    \draw[dashed] (-0.9473,1.3696) -- (-0.3,1.3696);
    \draw[->] (-0.9473,1.3696) -- (-0.3473,1.9696);
    \draw[->] (1.3696,-0.9473) -- (1.9696,-0.3473);
    \draw[->] (0.814,0.814) -- (1.414,1.414);

    \draw[->]
    (100:1.8)
    arc[start angle=100,end angle=48,radius=1.8];
    \draw[->]
    (-10:1.8)
    arc[start angle=-10,end angle=42,radius=1.8];

    \draw (-0.9473,1.3696) ++(0:0.3)
        arc[start angle=0,end angle=45,radius=0.3];
    \node at (-0.9473+0.4,1.3696+0.2) {\small$\theta$};
    \end{scope}
    \end{tikzpicture}
    \caption{When rotational diffusion is neglected, a particle with orientation $\theta=\pi/4$ first collides with the boundary and then moves along it. The particle can leave the inner boundary only when $\cos(\theta-\xi)=0$, but it cannot leave the outer boundary.}
    \label{fig:annulardemonstration}
\end{figure}
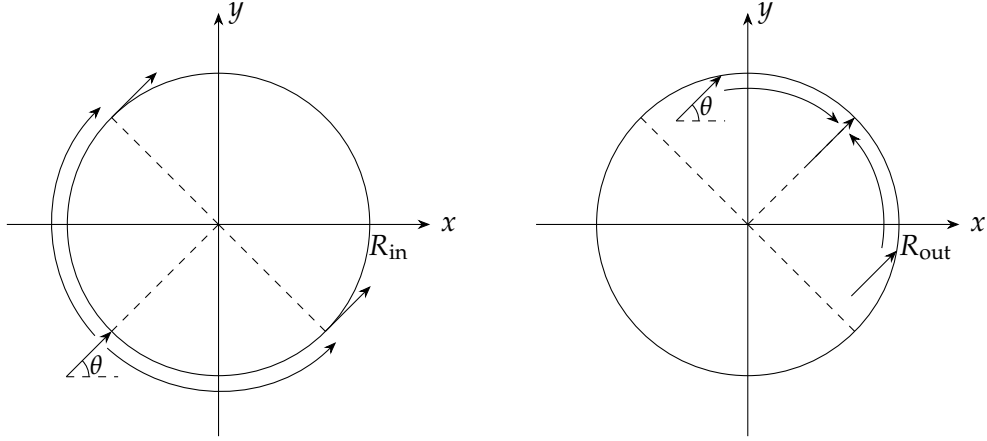

 When the particle moves freely inside the domain, the governing equation remains the same as for the circular case as illustrated in Figure \ref{fig:illustration_of_movement}. When the particle is blocked by the rigid boundaries,  $\dd r_t$ becomes $0$, while $\xi_t$ and $\theta_t$ change in the same way as when the particle is at free moving state. Therefore, the corresponding model is described by the following piecewise-deterministic Markov process \cite{davis1984piecewise} which reads
\begin{equation}
\left\{
\begin{aligned}
\dd r_t &= 
\Biggl\{
\begin{aligned}
& V \cos(\theta_t-\xi_t)\,\dd t, && (r_t,\xi_t,\theta_t)\in\Omega_{(r,\xi,\theta)},\\
& 0, && \text{otherwise,}
\end{aligned}
\Biggr.\\
\dd \xi_t &= \arctan\left(\frac{V\sin(\theta_t-\xi_t)\,\dd t}{V\cos(\theta_t-\xi_t)\,\dd t + r_t}\right),\\
\dd \theta_t &= \sqrt{2 D_\theta}\,\dd B_t,
\end{aligned}\label{eq:annularSDE}
\right.
\end{equation}
where $\Omega_{(r,\xi,\theta)} = \{ R_\mathrm{in}< r_t < R_\mathrm{out}\}\cup \{ r_t=R_\mathrm{out},\, \cos(\theta_t-\xi_t)<0 \}\cup \{ r_t=R_\mathrm{in},\, \cos(\theta_t-\xi_t)>0\}$.
Here the main difference is the physical domain. If we consider only the mathematical forms, the individual-based model \eqref{eq:annularSDE} has a form similar to \eqref{eq:circularSDE} for the circular domain. It is difficult to see the difference between concave and convex boundaries, since the difference is embedded in $\Omega_{(r,\xi,\theta)} $. Different boundaries have different ways to define motion toward or away from each boundary. 

\subsection{The Fokker-Planck Model}\label{ch:3.2}
We use $p(t,r,\xi,\theta)$ to describe the probability density in the interior of the domain at time $t$, position $(r,\xi)$ in polar coordinates and orientation $\theta$. As a new boundary is introduced, we use $q_\mathrm{out}(t,\xi,\theta)$ to denote the probability density at the outer boundary $\{r=R_\mathrm{out}\}$ of the annulus and introduce $q_\mathrm{in}(t,\xi,\theta)$ to denote the probability density at the inner boundary $\{r = R_\mathrm{in}\}$. The Fokker-Planck system then reads
\begin{equation}
\left\{
\begin{aligned}\label{eq:annularPDE}
    &\frac{\partial(rp)}{\partial t}+V\cos(\theta-\xi)\frac{\partial(rp)}{\partial r}
    +\frac{V}{r}\frac{\partial\left[\sin(\theta-\xi)rp\right]}{\partial \xi}
    -D_\theta\frac{\partial^2(rp)}{\partial \theta^2}=0,\quad r\in[R_\mathrm{in},R_\mathrm{out}],\\
    &\frac{\partial(R_\mathrm{out}q_\mathrm{out})}{\partial t}
    +\frac{V}{R_\mathrm{out}}\frac{\partial\left[\sin(\theta-\xi)R_\mathrm{out}q_\mathrm{out}\right]}{\partial \xi}
    -D_\theta\frac{\partial^2(R_\mathrm{out}q_\mathrm{out})}{\partial \theta^2}
    =V\cos(\theta-\xi)R_\mathrm{out}p(t,R_\mathrm{out},\xi,\theta),\\
    &\phantom{\frac{\partial(rp)}{\partial t}+V\cos(\theta-\xi)\frac{\partial(rp)}{\partial r}
    +\frac{V}{r}\frac{\partial\left[\sin(\theta-\xi)rp\right]}{\partial \xi}
    -D_\theta\frac{\partial^2(rp)}{\partial \theta^2}=0,}\quad \theta\in\left(\xi-\frac{\pi}{2},\xi+\frac{\pi}{2}\right),\\
    &\frac{\partial(R_\mathrm{in}q_\mathrm{in})}{\partial t}
    +\frac{V}{R_\mathrm{in}}\frac{\partial\left[\sin(\theta-\xi)R_\mathrm{in}q_\mathrm{in}\right]}{\partial \xi}
    -D_\theta\frac{\partial^2(R_\mathrm{in}q_\mathrm{in})}{\partial \theta^2}
    =-V\cos(\theta-\xi)R_\mathrm{in}p(t,R_\mathrm{in},\xi,\theta),\\
    &\phantom{\frac{\partial(rp)}{\partial t}+V\cos(\theta-\xi)\frac{\partial(rp)}{\partial r}
    +\frac{V}{r}\frac{\partial\left[\sin(\theta-\xi)rp\right]}{\partial \xi}
    -D_\theta\frac{\partial^2(rp)}{\partial \theta^2}=0,}\quad \theta\in\left(\xi+\frac{\pi}{2},\xi+\frac{3\pi}{2}\right).
\end{aligned}
\right.
\end{equation}
The coefficients in the PDE system above all share the same physical interpretations as the corresponding terms in (\ref{eq:PDEsystem}) for circular domains. More specifically, $V\cos(\theta-\xi)$ and $V\sin(\theta-\xi)$ represent the radial and tangential components of the velocity $V$. Tangential velocity is converted into angular velocity so $r$ appears in the denominator. The second-order derivative of $\theta$ takes into account the diffusion process in the orientation variable with strength $D_\theta$. The transition from free-swimming state to boundary-contact state is depicted by the source terms on the right-hand side of the equations for $R_\mathrm{out}q_\mathrm{out}$ and $R_\mathrm{in}q_\mathrm{in}$. 

The first two equations are the same as for the circular case, while the main difference is the additional one for the inner boundary, which can be considered as an obstacle inside the domain. For $q_\mathrm{in}(t,\xi,\theta)$, $\theta$ is defined only inside the interval $\left(\xi+\frac{\pi}{2},\xi+\frac{3\pi}{2}\right)$. This is because only inside this interval $\cos(\theta-\xi)<0$ and the orientation of the particle is toward the boundary, which is different from the outer boundary. Moreover, a negative sign appears on the right-hand side of the equation for $q_\mathrm{in}(t,\xi,\theta)$. It is easy to check that when $\theta\in\left(\xi+\frac{\pi}{2},\xi+\frac{3\pi}{2}\right)$, $\cos(\theta-\xi)<0$, thus the source term of $q_\mathrm{in}(t,\xi,\theta)$ remains positive, which takes into account the flux of those free-swimming particles that move toward the obstacle.

The Fokker-Planck model (\ref{eq:annularPDE}) admits similarities in the formulation of boundary conditions. One major difference, which is illustrated by Figure \ref{fig:annulardemonstration}, is that there are four critical points $(R_{\mathrm{in}},\xi,\xi\pm\pi/2)$ and $(R_{\mathrm{out}},\xi,\xi\pm\pi/2)$  instead of two. The boundary conditions are as follows.

\begin{itemize}[leftmargin=*]
    \item \textbf{Boundary conditions for $p(t,r,\xi,\theta)$.}
        
       The Dirichlet boundary condition analogous to (\ref{eq:circular0dirichlet}) is 
    \begin{equation}\label{eq:outer0boundary}
    p(t,R_\mathrm{out},\xi,\theta)=0,\quad \cos(\theta-\xi)<0,
    \end{equation}
     While for the inner boundary, since particles will leave the boundary immediately when its orientation points into the annular domain with $\cos(\theta-\xi)>0$, the corresponding boundary condition is
    \begin{equation}\label{eq:inner0boundary}
        p(t,R_\mathrm{in},\xi,\theta)=0,\quad \cos(\theta-\xi)>0.
    \end{equation}
    These two boundary conditions indicate that no particles will be at the boundaries when their velocities point into the domain. The periodic boundary conditions for $p$ in $\xi$ and $\theta$ from Section \ref{ch:2} remain unchanged:
    \begin{equation}
    \begin{aligned}
        &p(t,r,\xi,\theta)= p(t,r,\xi+2\pi,\theta),\\
        &p(t,r,\xi,\theta)=p(t,r,\xi,\theta+2\pi),\quad \frac{\partial}{\partial\theta} p(t,r,\xi,\theta)=\frac{\partial}{\partial\theta} p(t,r,\xi,\theta+2\pi).
    \end{aligned}
    \end{equation}

    \item \textbf{Boundary conditions for $q_\mathrm{out}(t,\xi,\theta)$ and $q_\mathrm{in}(t,\xi,\theta)$.}

    For $q_\mathrm{out}$ and $q_\mathrm{in}$, we continue to use periodic boundary conditions in $\xi$ direction and the Dirichlet conditions in $\theta$ direction. This is because the polar angle $\xi$ also admits periodicity in the annular domains as in the circular case. Additionally, particles leave the outer circle at $(t, R_\mathrm{out},\xi,\xi-\pi/2)$ and $(t, R_\mathrm{out},\xi,\xi+\pi/2)$, while they leave the inner circle at $(t, R_\mathrm{in},\xi,\xi+\pi/2)$ and $(t, R_\mathrm{in},\xi,\xi+3\pi/2)$. The corresponding boundary conditions read
    \begin{equation}
    \begin{aligned}
        &q_\mathrm{out}(t,\xi,\theta)=q_\mathrm{out}(t,\xi+2\pi,\theta),\quad &q_\mathrm{out}\left(t,\xi,\xi-\frac{\pi}{2}\right)=q_\mathrm{out}\left(t,\xi,\xi+\frac{\pi}{2}\right)=0,\\
        &q_\mathrm{in}(t,\xi,\theta)=q_\mathrm{in}(t,\xi+2\pi,\theta),\quad &q_\mathrm{in}\left(t,\xi,\xi+\frac{\pi}{2}\right)=q_\mathrm{in}\left(t,\xi,\xi+\frac{3\pi}{2}\right)=0.
    \end{aligned}  
    \end{equation}
   
    \item \textbf{Jump conditions at the transition interface.}
    
    To characterize the transition between free-swimming state inside the domain and boundary-contact state along the two boundaries, we need four boundary conditions at $(R_\mathrm{out}, \xi,\xi-\pi/2)$, $(R_\mathrm{out}, \xi,\xi+\pi/2)$, $(R_\mathrm{in}, \xi,\xi+\pi/2)$ and $(R_\mathrm{in}, \xi,\xi+3\pi/2)$ since these are the points where the radial component of the velocity vanishes. We integrate the system (\ref{eq:annularPDE}) as for the circular case and use the boundary conditions stated above to simplify the equations. The calculations are similar to those in Section \ref{ch:2.2} so the details are omitted here. As before, the boundary conditions are formulated in the sense of distributions and read
    \begin{equation}
    \left\{
    \begin{aligned}
    \frac{\partial}{\partial \theta}(rp)\bigg|_{\theta=(\xi-\pi/2)^-}^{\theta=(\xi-\pi/2)^+}
    =&
    -\frac{\partial}{\partial\theta}(R_\mathrm{out}q_\mathrm{out})\left(t,\xi,(\xi-\frac{\pi}{2})^+\right)\delta_{R_\mathrm{out}}(r)\\
    &+\frac{\partial}{\partial\theta}(R_\mathrm{in}q_\mathrm{in})\left(t,\xi,(\xi+\frac{3\pi}{2})^-\right)\delta_{R_\mathrm{in}}(r),\\
    \frac{\partial}{\partial \theta}(rp)\bigg|_{\theta=(\xi+\pi/2)^-}^{\theta=(\xi+\pi/2)^+}
    =&
    \frac{\partial}{\partial\theta}(R_\mathrm{out}q_\mathrm{out})\left(t,\xi,(\xi+\frac{\pi}{2})^-\right)\delta_{R_\mathrm{out}}(r)\\
    &-\frac{\partial}{\partial\theta}(R_\mathrm{in}q_\mathrm{in})\left(t,\xi,(\xi+\frac{\pi}{2})^+\right)\delta_{R_\mathrm{in}}(r).
    \end{aligned}\label{eq:annularjump}
    \right.
    \end{equation}
\end{itemize}

\paragraph{Mass Conservation.}
    The Fokker-Planck systems (\ref{eq:annularPDE}), together with the boundary conditions (\ref{eq:outer0boundary})-(\ref{eq:annularjump}), conserve the total mass, as stated in the following theorem.
    \begin{theorem}
        Consider the PDE system (\ref{eq:annularPDE}) with nonnegative bounded initial data
        \begin{equation}
            p(0,r,\xi,\theta) = p_0(r,\xi,\theta),\quad q_\mathrm{out}(0,\xi,\theta) = q_0^\mathrm{out}(\xi,\theta),\quad q_\mathrm{in}(0,\xi,\theta) = q_0^\mathrm{in}(\xi,\theta).
        \end{equation}
        Define the probability mass in the domain and along the boundaries to be
        \begin{equation}
    \begin{aligned}
    &M(t)= \int_{R_\mathrm{in}}^{R_\mathrm{out}}\int_{0}^{2\pi}\int_{\xi-\pi/2}^{\xi+3\pi/2}rp(t,r,\xi,\theta)\dd\theta\dd\xi\dd r,\\
    &M_\mathrm{out}(t)=\int_0^{2\pi}\int_{\xi-\pi/2}^{\xi+\pi/2}R_\mathrm{out}q_\mathrm{out}(t,\xi,\theta)\dd\theta\dd\xi,\\
     &M_\mathrm{in}(t)=\int_0^{2\pi}\int_{\xi+\pi/2}^{\xi+3\pi/2}R_\mathrm{in}q_\mathrm{in}(t,\xi,\theta)\dd\theta\dd\xi,
    \end{aligned}\label{eq:mass}
    \end{equation}
    respectively. Then the total mass $M(t)+M_\mathrm{out}(t)+M_\mathrm{in}(t)$ is a constant in time.
    \end{theorem}
    The proof is analogous to that of Theorem \ref{th:massconservation} and is therefore omitted here.

\paragraph{Radial-symmetric Model.}
    To derive the radial-symmetric model for an annular domain, we again introduce $\eta=\theta-\xi$ and absorb the eliminated angular coordinate into the reduced densities by setting $\tilde p(t,r,\eta)=2\pi p(t,r,\xi,\theta)$, $\tilde q_{\mathrm{out}}(t,\eta)=2\pi q_{\mathrm{out}}(t,\xi,\theta)$ and $\tilde q_{\mathrm{in}}(t,\eta)=2\pi q_{\mathrm{in}}(t,\xi,\theta)$. As in the circular-domain case, the factor $2\pi$ accounts for the integration over $\xi$. The reduced system then becomes
    \begin{equation}
    \left\{
    \begin{aligned}\label{eq:annularisotropicPDE}
    &\frac{\partial(r\tilde p)}{\partial t}+V\cos\eta\frac{\partial(r\tilde p)}{\partial r}
    -\frac{V}{r}\frac{\partial(\sin\eta\, r\tilde p)}{\partial \eta}
    -D_\theta\frac{\partial^2(r\tilde p)}{\partial \eta^2}=0,\quad r\in[R_\mathrm{in},R_\mathrm{out}],\\
    &\frac{\partial(R_\mathrm{out}\tilde q_\mathrm{out})}{\partial t}
    -\frac{V}{R_\mathrm{out}}\frac{\partial(\sin\eta\,R_\mathrm{out}\tilde q_\mathrm{out})}{\partial \eta}
    -D_\theta\frac{\partial^2(R_\mathrm{out}\tilde q_\mathrm{out})}{\partial \eta^2}
    =V\cos\eta\,R_\mathrm{out}\tilde p(t,R_\mathrm{out},\eta),\\
    &\phantom{\frac{\partial(R_\mathrm{in}\tilde q_\mathrm{in})}{\partial t}
    -\frac{V}{R_\mathrm{in}}\frac{\partial(\sin\eta\,R_\mathrm{in}\tilde q_\mathrm{in})}{\partial \eta}
    -D_\theta\frac{\partial^2(R_\mathrm{in}\tilde q_\mathrm{in})}{\partial \eta^2}
    =-V\cos\eta\,R_\mathrm{in}\tilde p(t,R_\mathrm{in},\eta),}\quad \eta\in\left(-\frac{\pi}{2},\frac{\pi}{2}\right),\\
    &\frac{\partial(R_\mathrm{in}\tilde q_\mathrm{in})}{\partial t}
    -\frac{V}{R_\mathrm{in}}\frac{\partial(\sin\eta\,R_\mathrm{in}\tilde q_\mathrm{in})}{\partial \eta}
    -D_\theta\frac{\partial^2(R_\mathrm{in}\tilde q_\mathrm{in})}{\partial \eta^2}
    =-V\cos\eta\,R_\mathrm{in}\tilde p(t,R_\mathrm{in},\eta),\quad \eta\in\left(\frac{\pi}{2},\frac{3\pi}{2}\right).
    \end{aligned}
    \right.
\end{equation}
The corresponding boundary conditions are
\begin{equation}
\left\{
\begin{aligned}
    &\tilde p(t,r,\eta) = \tilde p(t,r,\eta+2\pi),\quad \frac{\partial }{\partial\eta}\tilde p(t,r,\eta) = \frac{\partial }{\partial\eta}\tilde p(t,r,\eta+2\pi),\\
    &\tilde p(t,R_\mathrm{out},\eta) = 0,\quad \eta\in\left(\frac{\pi}{2},\frac{3\pi}{2}\right),\quad \tilde p(t,R_\mathrm{in},\eta) = 0,\quad \eta\in\left(-\frac{\pi}{2},\frac{\pi}{2}\right),\\
    &\tilde q_\mathrm{out}\left(t,-\frac{\pi}{2}\right)=\tilde q_\mathrm{out}\left(t,\frac{\pi}{2}\right)=0,\quad \tilde q_\mathrm{in}\left(t,\frac{\pi}{2}\right)=\tilde q_\mathrm{in}\left(t,\frac{3\pi}{2}\right)=0,\\
    &\frac{\partial}{\partial\eta}(r\tilde p)\bigg|_{(-\pi/2)^-}^{(-\pi/2)^+}=-\frac{\partial}{\partial\eta}(R_\mathrm{out}\tilde q_\mathrm{out})\left(t,(-\frac{\pi}{2})^+\right)\delta_{R_\mathrm{out}}(r)+\frac{\partial}{\partial\eta}(R_\mathrm{in}\tilde q_\mathrm{in})\left(t,(\frac{3\pi}{2})^-\right)\delta_{R_\mathrm{in}}(r),\\
    &\frac{\partial}{\partial\eta}(r\tilde p)\bigg|_{(\pi/2)^-}^{(\pi/2)^+}=\frac{\partial}{\partial\eta}(R_\mathrm{out}\tilde q_\mathrm{out})\left(t,(\frac{\pi}{2})^-\right)\delta_{R_\mathrm{out}}(r)-\frac{\partial}{\partial\eta}(R_\mathrm{in}\tilde q_\mathrm{in})\left(t,(\frac{\pi}{2})^+\right)\delta_{R_\mathrm{in}}(r).
    \end{aligned}
    \right.\label{radialsymmetricbound}
\end{equation}

\section{Numerical Discretization}\label{ch:4}

\subsection{Discretization of the PDE Model}\label{ch:4.3}
The radial-symmetric models in annular domains \eqref{eq:annularisotropicPDE}-\eqref{radialsymmetricbound} are used to demonstrate the numerical discretization in this section. For circular domains, similar discretizations can be used in a straightforward way. We adopt an upwind method in radial direction and an exponentially weighted conservative difference scheme in angular direction. An implicit time discretization is employed to ensure the stability of the scheme. For simplicity of the notations, we henceforth omit the tildes and, without relabeling, use $p$, $q_\mathrm{out}$, and $q_\mathrm{in}$ for $\tilde p$, $\tilde q_\mathrm{out}$, and $\tilde q_\mathrm{in}$, respectively.

Let $N_r$ be a positive integer and $N_\eta$ a positive even integer. We introduce the uniform mesh
\begin{equation}
\begin{aligned}
    &\Delta r=\frac{R_{\mathrm{out}}-R_{\mathrm{in}}}{N_r},\quad
    &&r_i=R_{\mathrm{in}}+\left(i-\frac12\right)\Delta r,\quad
    &i=1,\ldots,N_r,\\
    &\Delta\eta=\frac{2\pi}{N_\eta},\quad
    &&\eta_j=-\frac{\pi}{2}+\left(j-\frac12\right)\Delta\eta,\quad
    &j=1,\ldots,N_\eta.
\end{aligned}
\end{equation}
    The cell interfaces are located at $r_{i\pm1/2}=r_i\pm\Delta r/2$ and $\eta_{j\pm1/2}=\eta_j\pm\Delta\eta/2$. As illustrated in Figure \ref{fig:annulargrid}, the angular nodes on the outer and inner boundaries are
\begin{equation}
\eta_{\mathrm{out},m}=\eta_m,\quad\eta_{\mathrm{in},m}=\eta_{N_\eta/2+m},\quad m=1,\ldots,N_\eta/2.
\end{equation}

\begin{figure}[htbp]
\centering
\begin{tikzpicture}[x=0.8cm,y=0.65cm,>=stealth,font=\small]
    \draw[step=1,gray!30,very thin] (0,0) grid (6,8);
    \draw[black,thick] (0,0) rectangle (6,8);
    \draw[dashed,gray!75] (0,4) -- (6,4);
    \draw[blue!70!black,very thick] (0,4) -- (0,8);
    \draw[red!75!black,very thick] (6,0) -- (6,4);
    \foreach \i in {0,...,5} {
        \foreach \j in {0,...,7} {
            \fill (\i+0.5,\j+0.5) circle (1.25pt);
        }
    }
    \foreach \j in {0,...,3} {
        \fill[red!75!black] (6,\j+0.5) circle (2.1pt);
        \fill[blue!70!black] (0,\j+4.5) circle (2.1pt);
    }
    \foreach \y in {4,8} {
        \draw[blue!70!black,thick,fill=white] (0,\y) circle (2.1pt);
    }
    \foreach \y in {0,4} {
        \draw[red!75!black,thick,fill=white] (6,\y) circle (2.1pt);
    }
    \node[left=5pt] at (0,0) {$-\pi/2$};
    \node[left=5pt] at (0,4) {$\pi/2$};
    \node[left=5pt] at (0,8) {$3\pi/2$};
    \node[below=5pt] at (0,0) {$R_{\mathrm{in}}$};
    \node[below=5pt] at (6,0) {$R_{\mathrm{out}}$};
    \draw[->,thick] (0.5,-0.85) -- (5.5,-0.85);
    \node[below=3pt] at (0.5,-0.85) {$1$};
    \node[below=3pt] at (5.5,-0.85) {$N_r$};
    \draw[->,thick] (-1.65,0.5) -- (-1.65,7.5);
    \node[left=4pt] at (-1.65,0.5) {$1$};
    \node[left=4pt] at (-1.65,7.5) {$N_\eta$};
\end{tikzpicture}
\caption{Locations of the discrete unknowns in the $(r,\eta)$ domain. Black dots represent $p_{ij}$ at the interior cell centers $(r_i,\eta_j)$, red dots represent $q_{\mathrm{out},m}$ at $(R_{\mathrm{out}},\eta_m)$, and blue dots represent $q_{\mathrm{in},m}$ at $(R_{\mathrm{in}},\eta_{N_\eta/2+m})$. Open circles mark the endpoints where the density is given by homogeneous Dirichlet boundary condition and are not additional unknowns.}
\label{fig:annulargrid}
\end{figure}
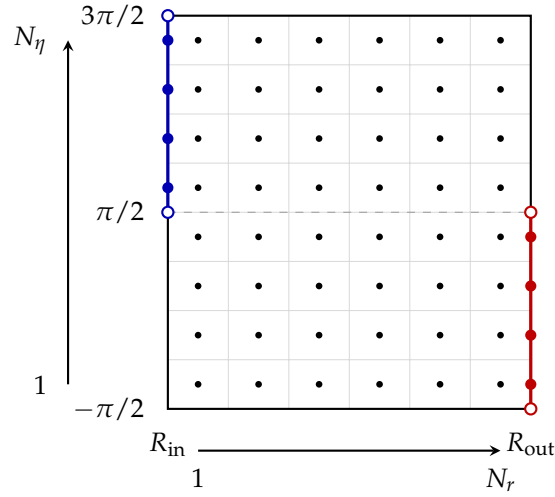

At time $t^n=n\Delta t$, the discrete unknowns approximate the cell averages of the Jacobian-weighted densities:
\begin{equation}
\begin{aligned}
    &p_{ij}^n\approx\frac{1}{\Delta r\Delta\eta}\int_{r_{i-1/2}}^{r_{i+1/2}}\int_{\eta_{j-1/2}}^{\eta_{j+1/2}}r p(t^n,r,\eta)\,\mathrm d\eta\,\mathrm dr,\\
    &q_{\mathrm{out},m}^n\approx\frac{1}{\Delta\eta}\int_{\eta_{m-1/2}}^{\eta_{m+1/2}}R_{\mathrm{out}}q_{\mathrm{out}}(t^n,\eta)\,\mathrm d\eta,\qquad
    q_{\mathrm{in},m}^n\approx\frac{1}{\Delta\eta}\int_{\eta_{N_\eta/2+m-1/2}}^{\eta_{N_\eta/2+m+1/2}}R_{\mathrm{in}}q_{\mathrm{in}}(t^n,\eta)\,\mathrm d\eta.
\end{aligned}
\label{eq:fdunknowns}
\end{equation}
Then the fully discrete scheme reads
\begin{equation}
\left\{
\begin{aligned}
    &\frac{p_{ij}^{n+1}-p_{ij}^{n}}{\Delta t}+V\cos\eta_j\frac{p_{ij}^{n+1}-p_{i-1,j}^{n+1}}{\Delta r}-D_\theta\frac{F_{i,j+1/2}^{n+1}-F_{i,j-1/2}^{n+1}}{\Delta\eta}=S_{ij}^{n+1},
    &&j=1,\ldots,N_\eta/2,\\
    &\frac{p_{ij}^{n+1}-p_{ij}^{n}}{\Delta t}+V\cos\eta_j\frac{p_{i+1,j}^{n+1}-p_{ij}^{n+1}}{\Delta r}-D_\theta\frac{F_{i,j+1/2}^{n+1}-F_{i,j-1/2}^{n+1}}{\Delta\eta}=S_{ij}^{n+1},
    &&j=N_\eta/2+1,\ldots,N_\eta,\\
    &\frac{q_{\mathrm{out},m}^{n+1}-q_{\mathrm{out},m}^{n}}{\Delta t}-D_\theta\frac{G_{\mathrm{out},m+1/2}^{n+1}-G_{\mathrm{out},m-1/2}^{n+1}}{\Delta\eta}
    =V\cos\eta_m\,p_{N_r,m}^{n+1},
    &&m=1,\ldots,N_\eta/2,\\
    &\frac{q_{\mathrm{in},m}^{n+1}-q_{\mathrm{in},m}^{n}}{\Delta t}
    -D_\theta\frac{G_{\mathrm{in},m+1/2}^{n+1}-G_{\mathrm{in},m-1/2}^{n+1}}{\Delta\eta}=-V\cos\eta_{N_\eta/2+m}\,p_{1,N_\eta/2+m}^{n+1},
    &&m=1,\ldots,N_\eta/2.
\end{aligned}
\right.
\label{eq:numericalscheme}
\end{equation}
In \eqref{eq:numericalscheme}, the first two equations hold for \(i=1,\ldots,N_r\). The difference between the first two equations in \eqref{eq:numericalscheme} lies in the approximation of the advection term in the \(r\)-direction, which is determined by the sign of \(\cos\eta_j\).
The flux appearing in these two equations is defined by
\begin{equation}
    F_{i,j+1/2}^{n+1}
    = e^{\frac{V\cos\eta_{j+1/2}}{r_iD_\theta}}
    \left(
    \frac{e^{-\frac{V\cos\eta_{j+1}}{r_iD_\theta}}p_{i,j+1}^{n+1}
    - e^{-\frac{V\cos\eta_j}{r_iD_\theta}}p_{i,j}^{n+1}}{\Delta\eta}
    \right),
\label{eq:angularfluxp}
\end{equation}
This flux is conservative for the angular drift and diffusion terms, and its form follows from the identity
\begin{equation}
    D_\theta\frac{\partial(rp)}{\partial\eta}
    + \frac{V}{r}\sin\eta\,rp
    = D_\theta e^{\frac{V\cos\eta}{rD_\theta}}
    \frac{\partial}{\partial\eta}
    \left(
    e^{-\frac{V\cos\eta}{rD_\theta}}rp
    \right).
\label{eq:angularidentity}
\end{equation}
The term $S_{ij}^{n+1}$ arises from the jump conditions (\ref{eq:annularjump}) and will be specified together with the discretization of the boundary conditions.

In the third equation of \eqref{eq:numericalscheme},
\begin{equation}
    G_{\mathrm{out},m+1/2}^{n+1} = e^{\frac{V\cos\eta_{\mathrm{out},m+1/2}}{R_{\mathrm{out}}D_\theta}}\left(\frac{e^{-\frac{V\cos\eta_{\mathrm{out},m+1}}{R_{\mathrm{out}}D_\theta}}q_{\mathrm{out},m+1}^{n+1}-e^{-\frac{V\cos\eta_{\mathrm{out},m}}{R_{\mathrm{out}}D_\theta}}q_{\mathrm{out},m}^{n+1}}{\Delta\eta}\right),\quad m=1,\ldots,N_\eta/2-1.
\label{eq:angularfluxq}
\end{equation}
with $\eta_{\mathrm{out},m+1/2}=\eta_{\mathrm{out},m}+\Delta\eta/2$.  It provides a discretization of the angular drift and diffusion terms in the second equation of \eqref{eq:annularisotropicPDE} in a similar way as in (\ref{eq:angularfluxp}).
The corresponding discretization of $q_\mathrm{in}$ is obtained in the same way.

The discretized boundary conditions are
\begin{equation}
\left\{
\begin{alignedat}{3}
    &p_{i,0}^{n}=p_{i,N_\eta}^{n},
    &\qquad &p_{i,N_\eta+1}^{n}=p_{i,1}^{n},
    &\qquad &i=1,\ldots,N_r,\\
    &p_{0,j}^{n}=0,
    & & 
    & &j=1,\ldots,N_\eta/2,\\
    &p_{N_r+1,j}^{n}=0,
    & &
    & &j=N_\eta/2+1,\ldots,N_\eta,\\
    &q_{\mathrm{out},0}^{n}=0,
    & &q_{\mathrm{out},N_\eta/2+1}^{n}=0,\\
    &q_{\mathrm{in},0}^{n}=0,
    & &q_{\mathrm{in},N_\eta/2+1}^{n}=0.
\end{alignedat}
\right.
\label{eq:fdboundary}
\end{equation}
The first condition imposes angular periodicity, while the next two are for (\ref{eq:outer0boundary}) and (\ref{eq:inner0boundary}). For the boundary densities, the homogeneous Dirichlet conditions are imposed directly in the angular fluxes where the indices $0$ and $N_\eta/2+1$ refer to the $\eta$-boundary of $q_\mathrm{in}$ and $q_\mathrm{out}$. 
The flux at $\cos\eta = 0$ is approximated in a similar way as in \eqref{eq:angularfluxq}. Using the homogeneous Dirichlet conditions in \eqref{eq:fdboundary} and noting that the distance between the endpoint and the adjacent cell center is $\Delta\eta/2$,
we have the following one-sided numerical fluxes
\begin{equation}
\begin{aligned}
    &G_{\mathrm{out},1/2}^{n+1} = \frac{2}{\Delta\eta}e^{\frac{V}{R_\mathrm{out}D_\theta}(\cos\eta_{\mathrm{out},1/2}-\cos\eta_{\mathrm{out},1})}q_{\mathrm{out},1}^{n+1},\\
    &G_{\mathrm{out},N_\eta/2+1/2}^{n+1} =-\frac{2}{\Delta\eta}e^{\frac{V}{R_\mathrm{out}D_\theta}(\cos\eta_{\mathrm{out},N_\eta/2+1/2}-\cos\eta_{\mathrm{out},N_\eta/2})}q_{\mathrm{out},N_\eta/2}^{n+1},\\
    &G_{\mathrm{in},1/2}^{n+1} = \frac{2}{\Delta\eta}e^{\frac{V}{R_\mathrm{in}D_\theta}(\cos\eta_{\mathrm{in},1/2}-\cos\eta_{\mathrm{in},1})}q_{\mathrm{in},1}^{n+1},\\
    &G_{\mathrm{in},N_\eta/2+1/2}^{n+1} =-\frac{2}{\Delta\eta}e^{\frac{V}{R_\mathrm{in}D_\theta}(\cos\eta_{\mathrm{in},N_\eta/2+1/2}-\cos\eta_{\mathrm{in},N_\eta/2})}q_{\mathrm{in},N_\eta/2}^{n+1}.
\end{aligned}
\label{eq:wallendpointflux}
\end{equation}
The jump conditions (\ref{eq:annularjump}) give the injection of particles from the boundaries into the bulk at $(R_\mathrm{in},\pm\pi/2)$ and $R_\mathrm{out},\pm\pi/2$. The probability injected from a boundary increases the probability in the receiving bulk cell, so a source term $S_{ij}^{n+1}$ is added to account for this:
\begin{equation}
\begin{aligned}
    S_{ij}^{n+1}&=\frac{D_\theta}{\Delta r\Delta\eta}\delta_{i,N_r}\left(\delta_{j,N_\eta}G_{\mathrm{out},1/2}^{n+1}-\delta_{j,N_\eta/2+1}G_{\mathrm{out},N_\eta/2+1/2}^{n+1}\right)\\
    &+\frac{D_\theta}{\Delta r\Delta\eta}\delta_{i,1}\left(\delta_{j,N_\eta/2}G_{\mathrm{in},1/2}^{n+1}-\delta_{j,1}G_{\mathrm{in},N_\eta/2+1/2}^{n+1}\right).
\end{aligned}
\label{eq:fdrelease}
\end{equation}
Here $\delta_{ij}$ denotes the Kronecker delta, with $\delta_{ij}=1$ if $i=j$ and $\delta_{ij}=0$ otherwise. The radial Dirac terms $\delta_{R_{\mathrm{out}}}(r)$ and $\delta_{R_{\mathrm{in}}}(r)$ are approximated by $\delta_{i,N_r}/\Delta r$ and $\delta_{i,1}/\Delta r$ respectively. The factor $1/\Delta\eta$ converts the release at each endpoint into a change in the average density of the receiving angular cell. 
\begin{remark}[First order approximation of the jump condition in $\eta$-direction]
Consider the outer boundary at $-\pi/2$ with homogeneous boundary condition $R_\mathrm{out}q_\mathrm{out}(t,-\pi/2)=0$, we have
\begin{equation}
\begin{aligned}
    G_{\mathrm{out},1/2}^{n+1}
    &=
    \frac{2}{\Delta\eta}e^{-\frac{V}{R_{\mathrm{out}}D_\theta}\sin\frac{\Delta\eta}{2}}\frac{1}{\Delta\eta}\int_{-\pi/2}^{-\pi/2+\Delta\eta}R_\mathrm{out}q_\mathrm{out}(t^{n+1},\eta)\dd\eta\\
    &=
    \frac{2}{\Delta\eta}\left(1-\frac{V}{2R_{\mathrm{out}}D_\theta}\Delta\eta+O(\Delta\eta^2)\right)\left(\frac{\Delta\eta}{2}\frac{\partial}{\partial\eta}(R_\mathrm{out}q_\mathrm{out})(t^{n+1},-\frac{\pi}{2})+O(\Delta\eta^2)\right)\\
    &=\frac{\partial}{\partial\eta}(R_\mathrm{out}q_\mathrm{out})(t^{n+1},-\frac{\pi}{2})+O(\Delta\eta)
\end{aligned}
\end{equation}
Thus, the endpoint flux is consistent with the homogeneous Dirichlet condition, and $D_\theta G_{\mathrm{out},1/2}^{n+1}$ approximates $D_\theta \partial(R_\mathrm{out}q_\mathrm{out})/\partial\eta$ with first-order accuracy. 
\end{remark}

\begin{remark}[Numerical discretization of the PDE system in circular domains.]
For the circular domain, the inner-boundary condition is replaced by transmission through the origin with a $\pi$-shift in $\eta$. Since the angular drift coefficient scales as $1/r$, it becomes large near the origin as the radial mesh is refined. We adopt a Scharfetter--Gummel discretization to handle this large angular drift. With mesh widths $\Delta r$ and $\Delta\eta$, the bulk angular flux is discretized as
\begin{equation}
    J_{i,j+\frac12}=\frac{D_\theta}{\Delta\eta}\left[B(z_{i,j+\frac12})(rp)_{i,j}-B(-z_{i,j+\frac12})(rp)_{i,j+1}\right],\qquad z_{i,j+\frac12}=\frac{V}{D_\theta r_i}(\cos\eta_j-\cos\eta_{j+1}).
\end{equation}
Here $B(z)=z/(e^z-1)$ and $B(0)=1$. For smooth solutions at fixed $r_i>0$, Taylor expansion gives
\begin{equation}
    J_{i,j+\frac12}=\left(-\frac{V}{r}\sin\eta\,(rp)-D_\theta\partial_\eta(rp)\right)\Bigg|_{(r_i,\eta_{j+\frac12})}+O(\Delta\eta^2).
\end{equation}
Thus, the angular flux is second-order consistent at fixed positive radii. Near the origin, the exponential factors in (\ref{eq:angularfluxp}) can produce excessively large coefficients. In contrast, $B(z)\sim ze^{-z}$ and $B(-z)\sim z$ as $z\to+\infty$, so the coefficients grow at most linearly in $|z|$. This avoids exponentially large flux coefficients near the origin. Finally, the angular part of the original equation is discretized by $(J_{i,j+\frac12}-J_{i,j-\frac12})/\Delta\eta$ and the remaining discretization follows the annular case.
\end{remark}

\subsection{Conservation of the Discretization}\label{ch:4.2}
We first define the discrete total mass at time $t^n = n\Delta t$ by
\begin{equation}
    M_h^n=\Delta r\Delta\eta\sum_{i=1}^{N_r}\sum_{j=1}^{N_\eta}p_{ij}^n+\Delta\eta\sum_{m=1}^{N_\eta/2}q_{\mathrm{out},m}^n+\Delta\eta\sum_{m=1}^{N_\eta/2}q_{\mathrm{in},m}^n,
\label{eq:discretemass}
\end{equation}
which approximates the continuous total mass
\begin{equation}
    M(t)=\int_{R_{\mathrm{in}}}^{R_{\mathrm{out}}}\int_{-\pi/2}^{3\pi/2}rp(t,r,\eta)\,\mathrm d\eta\,\mathrm dr+\int_{-\pi/2}^{\pi/2}R_{\mathrm{out}}q_{\mathrm{out}}(t,\eta)\,\mathrm d\eta+\int_{\pi/2}^{3\pi/2}R_{\mathrm{in}}q_{\mathrm{in}}(t,\eta)\,\mathrm d\eta.
\end{equation}
We verify that this quantity is conserved by the numerical scheme.

\begin{theorem}\label{th:numericalconservation}
    The discrete total mass (\ref{eq:discretemass}) satisfies $M_h^{n+1}=M_h^n$ for all $n\geq0$.
\end{theorem}
\begin{proof}
    Let the bulk and boundary masses be
    \begin{equation}
        M_{p,h}^n=\Delta r\Delta\eta\sum_{i=1}^{N_r}\sum_{j=1}^{N_\eta}p_{ij}^n,\quad 
        M_{\mathrm{out},h}^n=\Delta\eta\sum_{m=1}^{N_\eta/2}q_{\mathrm{out},m}^n,\quad M_{\mathrm{in},h}^n=\Delta\eta\sum_{m=1}^{N_\eta/2}q_{\mathrm{in},m}^n.
    \end{equation}

    First, we multiply the equations for $p$ by $\Delta r\Delta\eta$ and take summation over all interior cells. Each angular flux appears twice with opposite signs, and the two remaining endpoint fluxes cancel by periodicity, that is $\sum_{j=1}^{N_\eta}\left(F_{i,j+1/2}^{n+1}-F_{i,j-1/2}^{n+1}\right)=F_{i,N_\eta+1/2}^{n+1}-F_{i,1/2}^{n+1}=0.$
    The radial contributions from adjacent cells also cancel, leaving only the boundary terms. Using the zero conditions at the boundaries for density $p$, we obtain
    \begin{equation}
    \begin{aligned}
        &C_{\mathrm{out}}^{n+1}:=\Delta\eta\sum_{j=1}^{N_\eta/2}V\cos\eta_j\sum_{i=1}^{N_r}\left(p_{ij}^{n+1}-p_{i-1,j}^{n+1}\right)=\Delta\eta\sum_{j=1}^{N_\eta/2}V\cos\eta_j\,p_{N_r,j}^{n+1},\\
        &C_{\mathrm{in}}^{n+1}:=\Delta\eta\sum_{j=N_\eta/2+1}^{N_\eta}V\cos\eta_j\sum_{i=1}^{N_r}\left(p_{i+1,j}^{n+1}-p_{ij}^{n+1}\right)=-\Delta\eta\sum_{j=N_\eta/2+1}^{N_\eta}V\cos\eta_j\,p_{1,j}^{n+1}.
    \end{aligned}
    \end{equation}
    Thus the difference of the discrete bulk mass reads
    \begin{equation}
        \frac{M_{p,h}^{n+1}-M_{p,h}^{n}}{\Delta t}=-C_{\mathrm{out}}^{n+1}-C_{\mathrm{in}}^{n+1}+\Delta r\Delta\eta\sum_{i=1}^{N_r}\sum_{j=1}^{N_\eta}S_{ij}^{n+1}.
    \label{eq:interiormassbalance}
    \end{equation}
    For the outer boundary, let $E_\mathrm{out}^{n+1} := D_\theta\left(G_{\mathrm{out},1/2}^{n+1}-G_{\mathrm{out},N_\eta/2+1/2}^{n+1}\right)$ and define $E_\mathrm{in}^{n+1}$ in the same way for the inner wall. From (\ref{eq:fdrelease}), we have $\Delta r\Delta\eta\sum_{i=1}^{N_r}\sum_{j=1}^{N_\eta}S_{ij}^{n+1} =E_{\mathrm{out}}^{n+1}+E_{\mathrm{in}}^{n+1}$.
    Substituting this identity into (\ref{eq:interiormassbalance}) gives
    \begin{equation}
        \frac{M_{p,h}^{n+1}-M_{p,h}^{n}}{\Delta t}=-C_{\mathrm{out}}^{n+1}-C_{\mathrm{in}}^{n+1}+E_{\mathrm{out}}^{n+1}+E_{\mathrm{in}}^{n+1}.
    \label{eq:bulkmassbalance}
    \end{equation}

    Next, we multiply each boundary equation by $\Delta\eta$ and take summation over $m=1,\ldots,N_\eta/2$. The fluxes between adjacent cells cancel, leaving only the two endpoint fluxes, thus
    \begin{equation}
        \frac{M_{\mathrm{out},h}^{n+1}-M_{\mathrm{out},h}^{n}}{\Delta t}
        =C_{\mathrm{out}}^{n+1}-E_{\mathrm{out}}^{n+1},\quad \frac{M_{\mathrm{in},h}^{n+1}-M_{\mathrm{in},h}^{n}}{\Delta t}
        =C_{\mathrm{in}}^{n+1}-E_{\mathrm{in}}^{n+1}.
    \label{eq:wallmassbalance}
    \end{equation}
    Adding (\ref{eq:bulkmassbalance}) and (\ref{eq:wallmassbalance}), we finally get $(M_h^{n+1}-M_h^n)/\Delta t=0$, which shows that the scheme preserves total mass in the discrete sense.
\end{proof}

\section{Numerical Results}\label{ch:5}

\subsection{Performance of the Numerical Discretization}
To demonstrate the performance of the numerical discretization, we consider an example in an annular domain with $R_\mathrm{in}=1$, $R_\mathrm{out}=2$, $V=1$ and $D_\theta=1$. The initial particle distribution is uniform within the annulus and the total mass is normalized to 1.

\paragraph{Mass Conservation.}  We performed the simulation up to $t=10$ and the evolution of the total mass with respect to time is shown in Figure \ref{fig:massconservation}. The total mass remains constant throughout the computation, resulting in a horizontal line in the figure, which confirms the conservative property of the proposed numerical discretization as proved in section \ref{ch:4.2}.
\begin{figure}[!tbh]
    \centering
    \includegraphics[scale=0.4]{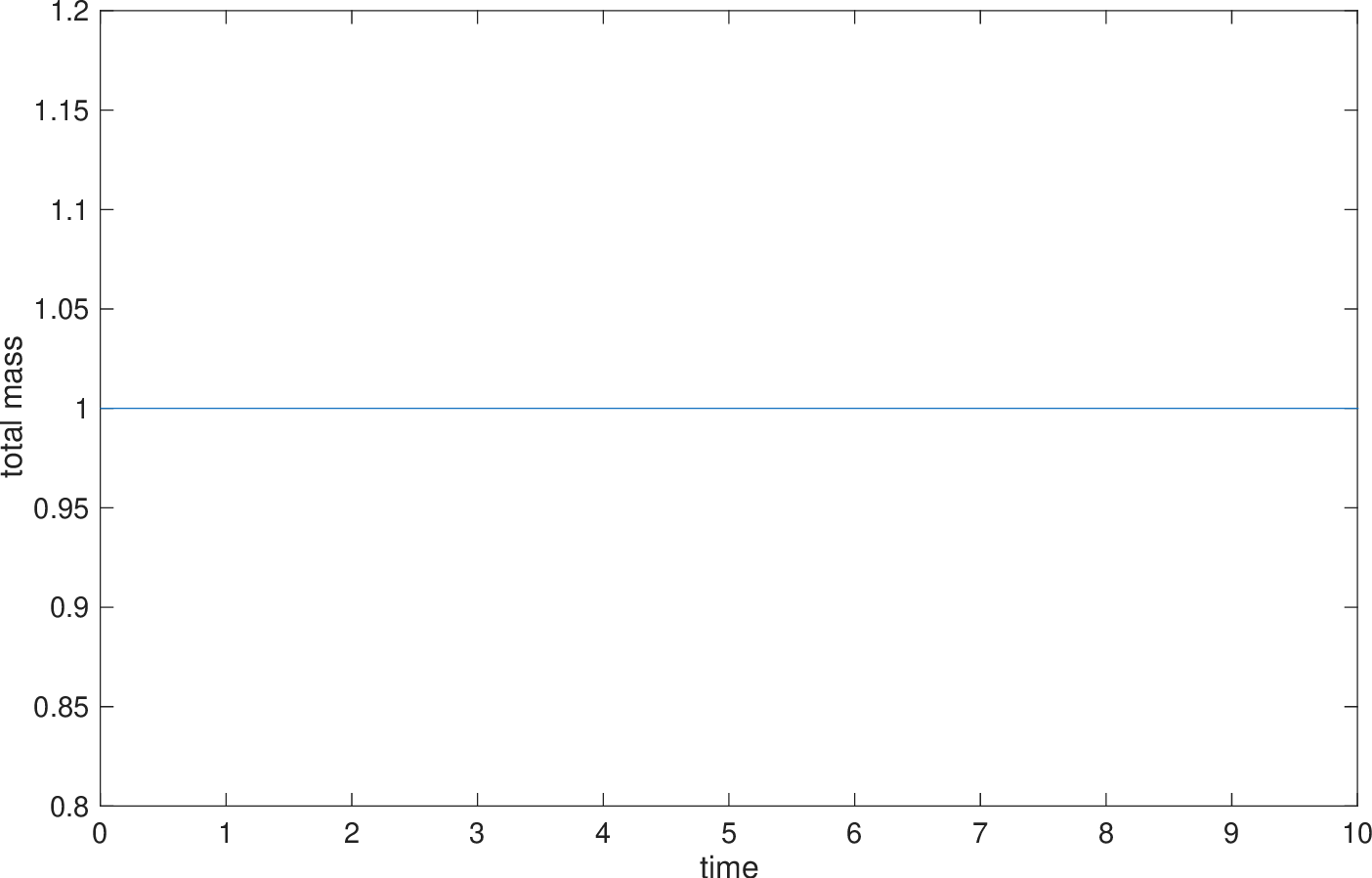}
    \caption{The time evolution of the discrete total mass defined in \eqref{eq:discretemass}.}\label{fig:massconservation}
\end{figure}

\paragraph{Convergence Order.}
To assess convergence, we compare the numerical solutions on successively refined grids with a reference solution computed on a $512\times512$ grid.
All solutions use the same time step $\Delta t=10^{-3}$ and are compared at $T=5$.
Each coarse-grid solution is prolonged to the reference grid by a piecewise-constant reconstruction that preserves the mass of each component.
The errors in the conserved densities are
\begin{equation}
\begin{aligned}
    e_p
    =& (rp)_h-(rp)_{\mathrm{ref}},\\
    e_{\mathrm{out}}
    =& (R_{\mathrm{out}}q_{\mathrm{out}})_h
       -(R_{\mathrm{out}}q_{\mathrm{out}})_{\mathrm{ref}},\\
    e_{\mathrm{in}}
    =& (R_{\mathrm{in}}q_{\mathrm{in}})_h
       -(R_{\mathrm{in}}q_{\mathrm{in}})_{\mathrm{ref}}.
\end{aligned}
\end{equation}
Here the subscript $h$ denotes the reconstructed coarse-grid quantities.

The radial Dirac distributions in the jump conditions \eqref{radialsymmetricbound} motivate assessing the error through integrals against test functions.
For the bulk error on $\Omega=(R_{\mathrm{in}},R_{\mathrm{out}})\times\mathbb{T}_{2\pi}$, the integral $\int_\Omega e_p\phi\dd r\dd\eta$ measures a weighted difference between the two densities; choosing $\phi=1$ gives their difference in total bulk mass.
To compare these integrals over different test functions while controlling their size and variation, we use
\begin{equation}
    \|e_p\|_{(H^1_{\mathrm{per}}(\Omega))'}
    =\sup_{0\ne\phi\in H^1_{\mathrm{per}}(\Omega)}
    \frac{\left|\int_\Omega e_p\phi\dd r\dd\eta\right|}
    {\left[\int_\Omega
    \left(|\phi|^2+|\partial_r\phi|^2+|\partial_\eta\phi|^2\right)
    \dd r\dd\eta\right]^{1/2}}.
\end{equation}
This is the dual norm of $H^1_{\mathrm{per}}(\Omega)$, whose functions are periodic in $\eta$ and have square-integrable values and first derivatives.
Since $\|e_p\|_{(H^1_{\mathrm{per}}(\Omega))'}\le\|e_p\|_{L^2(\Omega)}$, it provides a weaker error measure than the $L^2$ norm.

The two boundary errors are treated in the same way on $I_{\mathrm{out}}=(-\pi/2,\pi/2)$ and $I_{\mathrm{in}}=(\pi/2,3\pi/2)$.
For $b\in\{\mathrm{out},\mathrm{in}\}$, we choose test functions in $H^1_0(I_{b})$, which vanish at the angular endpoints, consistent with the homogeneous endpoint conditions for the boundary densities.
For $e=(e_p,e_{\mathrm{out}},e_{\mathrm{in}})$, the squared norms of the three components are added to give
\begin{equation}
\begin{aligned}
    X=&H^1_{\mathrm{per}}(\Omega)
    \times H^1_0(I_{\mathrm{out}})\times H^1_0(I_{\mathrm{in}}),\\
    \|e\|_{X'}^2
    =&\|e_p\|_{(H^1_{\mathrm{per}}(\Omega))'}^2
    +\|e_{out}\|_{(H^1_0(I_{out}))'}^2+\|e_{in}\|_{(H^1_0(I_{in}))'}^2.
\end{aligned}
\end{equation}
The dual norms can be evaluated by solving the auxiliary problems described in Remark~\ref{rmk:convergence}.
Using the same notation for the discrete error vectors, the resulting expression on the reference grid is
\begin{equation}
\begin{aligned}
    \|e\|_{X',\mathrm{disc}}^2
    =&
    \Delta r_{\mathrm{ref}}\Delta\eta_{\mathrm{ref}}\,
    e_p^\top A_p^{-1}e_p+
    \Delta\eta_{\mathrm{ref}}
    \left(
        e_{\mathrm{out}}^\top A_\eta^{-1}e_{\mathrm{out}}
        +
        e_{\mathrm{in}}^\top A_\eta^{-1}e_{\mathrm{in}}
    \right).
\end{aligned}\label{eq:discretenorm}
\end{equation}
Here $A_p$ discretizes $I-\partial_{rr}-\partial_{\eta\eta}$ with periodic angular and homogeneous Neumann radial boundary conditions, and $A_\eta$ discretizes $I-\partial_{\eta\eta}$ with homogeneous Dirichlet endpoint conditions.
The factors $\Delta r_{\mathrm{ref}}\Delta\eta_{\mathrm{ref}}$ and $\Delta\eta_{\mathrm{ref}}$ are the cell measures used to approximate the bulk and boundary integrals, respectively.

Table~\ref{tb:annuluserror} reports the errors for $N_r,N_\eta\in\{32,64,128,256\}$.
For the convergence plot in Figure~\ref{fig:annuluserror}, we choose $N_r=N_\eta=N$.
Along this refinement path, both mesh sizes are successively halved, and the error relative to the reference solution exhibits approximately first-order decay in the chosen discrete dual norm over the tested grid range.

\begin{table}[!tbh]
	\caption{Discrete dual-norm error for results in annular domains under different resolutions.}
	\label{tb:annuluserror}
	\centering
	\medskip\small\renewcommand{\arraystretch}{1.15}
	\begin{tabular}{||l|cccc||}
		\hline
		\diagbox{$N_r$}{$N_\eta$} & 32 & 64 & 128 & 256 \\
		\hline
		32  & 2.75e-02 & 2.05e-02 & 1.71e-02 & 1.55e-02\\
		64  & 2.46e-02 & 1.62e-02 & 1.18e-02 & 9.85e-03\\
		128 & 2.35e-02 & 1.39e-02 & 8.58e-03 & 5.98e-03\\
		256 & 2.32e-02 & 1.30e-02 & 6.75e-03 & 3.43e-03\\
		\hline
	\end{tabular}
\end{table}
\begin{figure}[!tbh]
	\centering
	\includegraphics[scale=0.5]{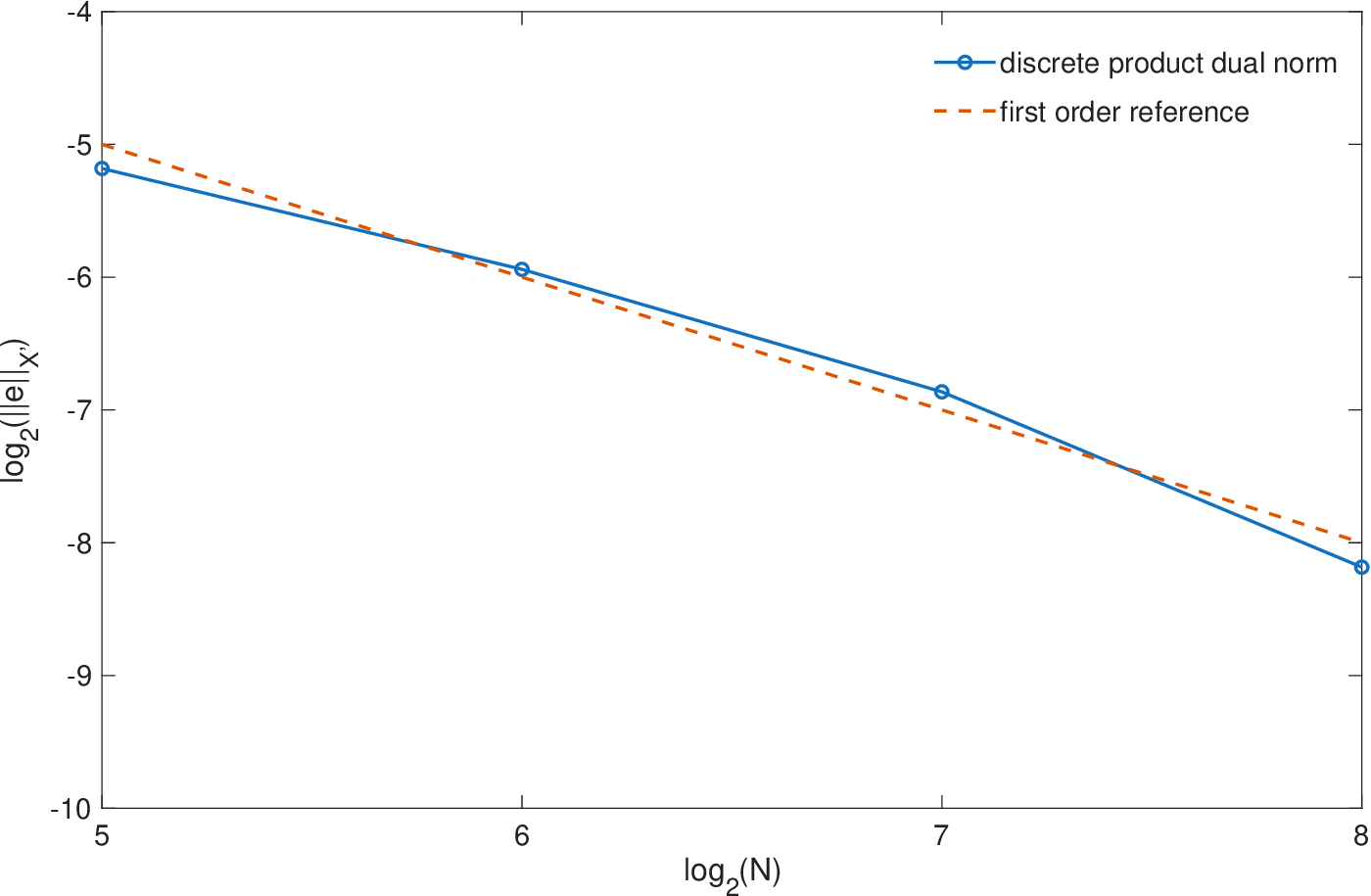}
	\caption{The convergence order with $N$.}\label{fig:annuluserror} 
\end{figure}

\begin{remark}[Evaluation of the dual norms]\label{rmk:convergence}
The bulk auxiliary solution $z_p\in H^1_{\mathrm{per}}(\Omega)$ is defined by
\begin{equation}
    \int_\Omega
    \left(z_p\phi+\partial_r z_p\partial_r\phi
    +\partial_\eta z_p\partial_\eta\phi\right)\dd r\dd\eta
    =\int_\Omega e_p\phi\dd r\dd\eta
\end{equation}
for all $\phi\in H^1_{\mathrm{per}}(\Omega)$ by the Riesz Representation Theorem.
This is the weak formulation of
\begin{equation}
\begin{cases}
    (I-\partial_{rr}-\partial_{\eta\eta})z_p=e_p
        &\text{in }\Omega,\\
    \partial_r z_p=0
        &\text{at }r=R_{\mathrm{in}},R_{\mathrm{out}},\\
    z_p(r,\eta+2\pi)=z_p(r,\eta).
\end{cases}
\end{equation}
The radial Neumann conditions arise because no radial boundary values are imposed on the test functions.
The Cauchy--Schwarz inequality bounds the normalized integral in the definition of the dual norm by $\|z_p\|_{H^1(\Omega)}$, with equality for $\phi=z_p$ when $z_p\ne0$.
Hence
\begin{equation}
    \|e_p\|_{(H^1_{\mathrm{per}}(\Omega))'}^2
    =\|z_p\|_{H^1(\Omega)}^2
    =\int_\Omega e_pz_p\dd r\dd\eta.
\end{equation}
Similarly, for $b\in\{\mathrm{out},\mathrm{in}\}$, the boundary auxiliary solution satisfies
\begin{equation}
\begin{cases}
    (I-\partial_{\eta\eta})z_b=e_b&\text{in }I_b,\\
    z_b=0&\text{on }\partial I_b,
\end{cases}
\end{equation}
and $\|e_b\|_{(H^1_0(I_b))'}^2=\int_{I_b}e_bz_b\dd\eta$.
Discretizing these problems gives $A_pz_p=e_p$ and $A_\eta z_b=e_b$; substituting their solutions into the integrals yields \eqref{eq:discretenorm}.
\end{remark}

\paragraph{Comparison with individual-based model in a circular domain}\label{ch:5.2.1}
We perform Monte-Carlo simulations for the individual-based model (\ref{eq:circularSDE}) with Euler-Maruyama-like scheme \cite{kloeden1992stochastic} after introducing $\eta = \theta-\xi$. Since $\eta$ is a linear combination of $\theta$ and $\xi$, $\dd\eta = \dd\theta - \dd\xi$. Note that $-B_t$ is also a standard Brownian motion and is therefore equivalent to $B_t$. The individual based model for the radial-symmetric model \eqref{eq:2dsystem} becomes 
\begin{equation}
\left\{
\begin{aligned}
\dd r_t &= 
\Biggl\{
\begin{aligned}
& V \cos\eta_t\,\dd t, && r_t < R\,\text{ or }\,r_t=R,\,\cos\eta_t<0,\\
& 0, && \text{otherwise,}
\end{aligned}
\Biggr.\\
\dd \eta_t &= -\arctan\left(\frac{V\sin\eta_t\,\dd t}{V\cos\eta_t\,\dd t + r_t}\right)+\sqrt{2 D_\theta}\,\dd B_t.
\end{aligned}
\right.
\end{equation}
We simulate $5\times10^6$ particles, and each of them is represented by $r_k^n$ and $\eta_k^n$ at $t^n = n\Delta t$. The pair $(r_k,\eta_k)$ is evolved by the following calculations:
\begin{enumerate}
    \item Compute $\eta_{k,\text{temp}}^{n+1}$. Generate an increment $Z_k^n\sim\mathcal{N}(0,1)$ independent of previous ones, set $\eta_{k,\text{temp}}^{n+1}=\eta_k^n-\arctan\left(V\sin\eta_k^n\Delta t/(V\cos\eta_k^n\Delta t + r_k^n)\right)+\sqrt{2 D_\theta\Delta t}\,Z_k^n$.
    \item Update $r_k^{n+1}$ and $\eta_k^{n+1}$. Set $r_k^{n+1}=\min\{|r_k^n+V\cos\eta_{k,\text{temp}}^{n+1}\Delta t|,R\}$. If $r_k^n+V\cos\eta_{k,\text{temp}}^{n+1}\Delta t<0$, then set $\eta_k^{n+1}=\eta_{k,\text{temp}}^{n+1}+\pi$. Otherwise, set $\eta_k^{n+1}=\eta_{k,\text{temp}}^{n+1}$.
\end{enumerate}

Figure \ref{fig:circularp} shows the stochastic simulation results and numerical results for the probability density function $rp(t,r,\eta)$ inside the domain. To accurately resolve the boundary-layer structure, we use 1000 grid points in the radial direction. For visualization, the numerical solution is averaged over each group of ten consecutive radial grid points, yielding a representation on a $100 \times 100$ grid. The numerical solutions are in good agreement with the particle simulation results, and correctly describe the vacuum region near the boundary when $\cos\eta<0$, which supports the validity of the Fokker-Planck model. The density peaks at $(R,\pm\pi/2)$ since these are the two points that particles at boundary contact state can transit to the free-swimming state.
As $t$ increases, the probability density in $\{0<r<R\}$ decreases while that near $\{r=R\}$ increases, indicating the emergence of boundary accumulation.

\begin{figure}[!tbh]
    \centering
    \includegraphics[width = 1.0\textwidth]{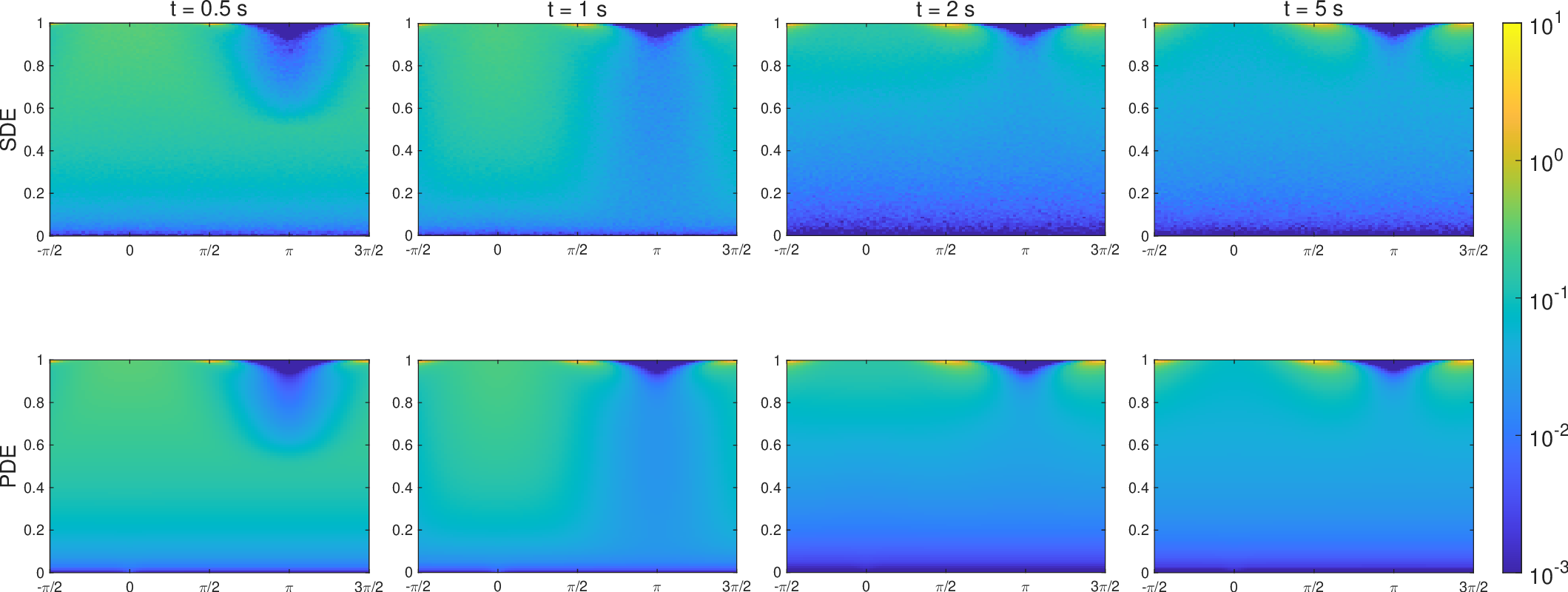}
    \caption{Comparison between SDE simulation results (top) and PDE numerical results (bottom) of the probability density $rp(t,r,\eta)$ inside a circular domain. From left to right: $t=0.5s$, $1s$, $2s$, $5s$.}\label{fig:circularp}
\end{figure}

To verify the behavior of $rp(t,r,\eta)$ near the boundary, $p_{N_r,j}$ are plotted in Figure \ref{fig:circularq}. Our numerical results do capture the peaks at $(R, \pm\pi/2)$ which get sharper as time increases, exhibiting the same behavior as SDE results. The results of $Rq(t,\eta)$ are also plotted in Figure \ref{fig:circularq}. The maximum of this function keeps increasing with time until the system itself reaches the steady state.
\begin{figure}[!tbh]
    \centering
    \includegraphics[width=1.0\textwidth]{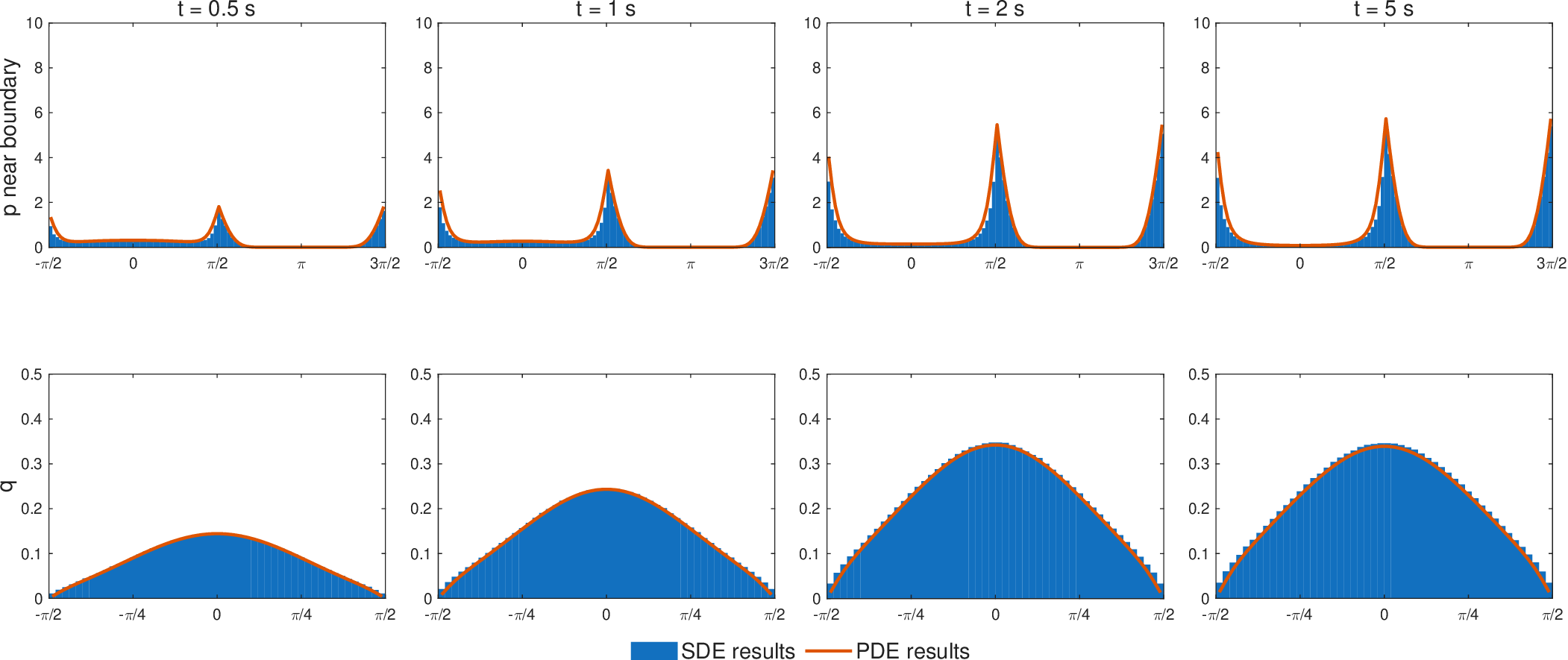}
    \caption{Comparison between SDE simulation results (bar) and PDE numerical results (curve) in a circular domain. Probability density near the boundary (Top) and on the boundary (Bottom). From left to right: $t=0.5s$, $1s$, $2s$, $5s$.}\label{fig:circularq}
\end{figure}

\paragraph{Comparison with individual-based model in an annular domain.}
We choose the radius of the outer boundary and inner boundary to be $R_\mathrm{out} = 2$ and $R_\mathrm{in} = 1$, respectively. Velocity $V = 1$ and rotational diffusion coefficient $D_\theta = 1$ in both stochastic simulations and numerical solutions. The initial condition is set to be the uniform distribution in $\{R_\mathrm{in}<r<R_\mathrm{out}\}$. All the results illustrated in this section share the same initial distribution and parameter settings.

We perform Monte-Carlo simulations for the individual-based model (\ref{eq:annularSDE}) with Euler-Maruyama-like scheme \cite{kloeden1992stochastic} after introducing $\eta = \theta-\xi$. The simplified individual-based model reads
\begin{equation}
\left\{
\begin{aligned}
\dd r_t &= 
\Biggl\{
\begin{aligned}
& V \cos\eta_t\,\dd t, && R_\mathrm{in} < r_t < R_\mathrm{out}\,\text{ or }\,r_t=R_\mathrm{out},\,\cos\eta_t<0 \,\text{ or }\,r_t=R_\mathrm{in},\,\cos\eta_t>0,\\
& 0, && \text{otherwise,}
\end{aligned}
\Biggr.\\
\dd\eta_t &= -\arctan\left(\frac{V\sin\eta_t\,\dd t}{V\cos\eta_t\,\dd t + r_t}\right) + \sqrt{2 D_\theta}\,\dd B_t.
\end{aligned}
\right.
\end{equation}
For each of the $5\times10^6$ particles, we use $r_k^n$ and $\eta_k^n$ to determine its state at $t^n = n\Delta t$. The pair $(r_k,\eta_k)$ is evolved by the following calculations:
\begin{enumerate}
    \item Update $\eta_k^n$. Generate an increment $Z_k^n   \sim\mathcal{N}(0,1)$ independent of previous ones, set $\eta_{k}^{n+1}=\eta_k^n-\arctan\left(V\sin\eta_t\Delta t/(V\cos\eta_k^n\Delta t + r_k^n)\right)+\sqrt{2 D_\theta\Delta t}\,Z_k^n$.
    \item Update $r_k^n $. Set $r_k^{n+1}=\max\{\min\{r_k^n+V\cos\eta_{k}^{n+1}\Delta t,R_\mathrm{out}\}, R_\mathrm{in}\}$.
\end{enumerate}

Figure \ref{fig:annulusp} shows the stochastic simulations and numerical results for the probability density $rp(t,r,\eta)$ inside the annulus at different observation times. Same as before, we use $1000$ grid points in the radial direction.  As time goes on, the density function peaks at $(R_\mathrm{out}, \pm\pi/2)$ and $(R_\mathrm{in}, \pm\pi/2)$ since these are the points where the particles switch between free-swimming and boundary aggregation states. The particle density in the interior of the domain decreases until it reaches the steady state, while the particles gradually accumulate at the two boundaries.
\begin{figure}[htbp]
    \centering
    \includegraphics[width=1.0\textwidth]{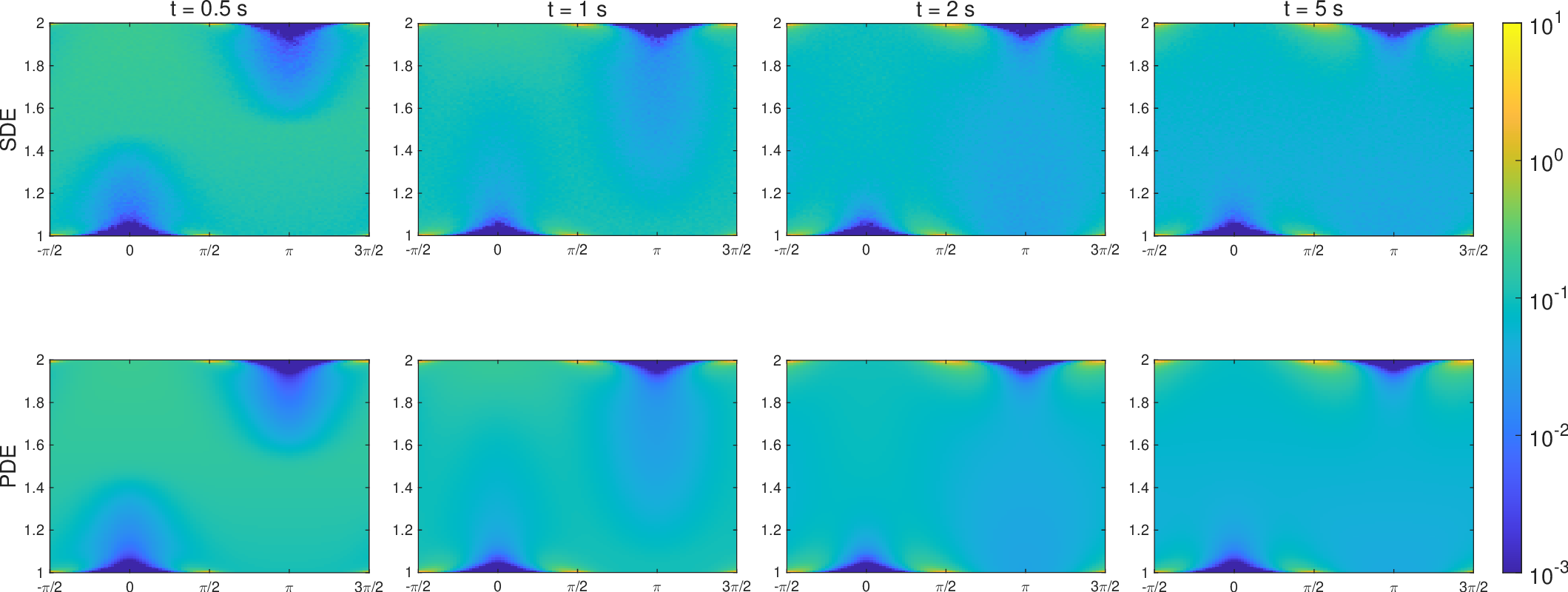}
    \caption{Comparison between SDE simulation results (top) and PDE numerical results (bottom) of the probability density $rp(t,r,\eta)$ inside an annular domain. From left to right: $t=0.5s$, $1s$, $2s$, $5s$.}\label{fig:annulusp}
\end{figure}

We plot $p_{1,j}$ and $p_{N_r, j}$ in Figure \ref{fig:annuluspend} to show the behavior of $rp(t,r,\eta)$ near the two boundaries of the domain. Our model successfully captured the peaks at $(R_\mathrm{in}, \pm\pi/2)$ and $(R_\mathrm{out}, \pm\pi/2)$. As is in the case discussed in section \ref{ch:5.2.1}, the peaks exhibited near the outer boundary get sharper as $t$ increases. However, a decrease in the height of the peaks near the inner boundary can be observed. We also plotted the numerical results of $R_\mathrm{in}q_\mathrm{in}$ and $R_\mathrm{out}q_\mathrm{out}$ at the selected time to show the evolution of these two probability densities in Figure \ref{fig:annulusq}. The probability of finding particles near the outer boundary is larger than that near the inner boundary, suggesting some kind of preferences when choosing the boundaries to aggregate.

\begin{figure}[!tbh]
    \centering
    \includegraphics[width=1.0\textwidth]{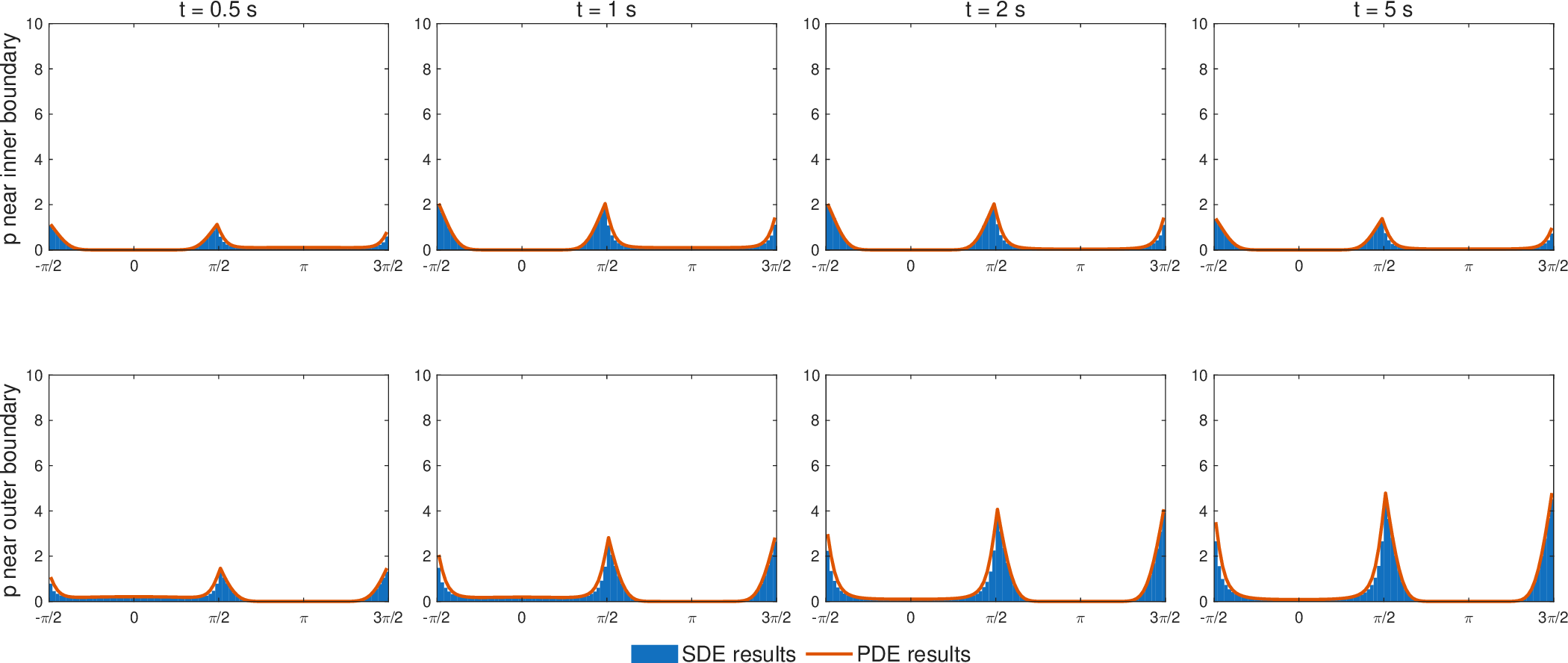}
    \caption{Comparison between SDE simulation results (bar) and PDE numerical results (curve) in an annular domain. Probability density near the inner circle (Top) and the outer circle (Bottom). From left to right: $t=0.5s$, $1s$, $2s$, $5s$.}\label{fig:annuluspend}
\end{figure}

\begin{figure}[!tbh]
    \centering
    \includegraphics[width=1.0\textwidth]{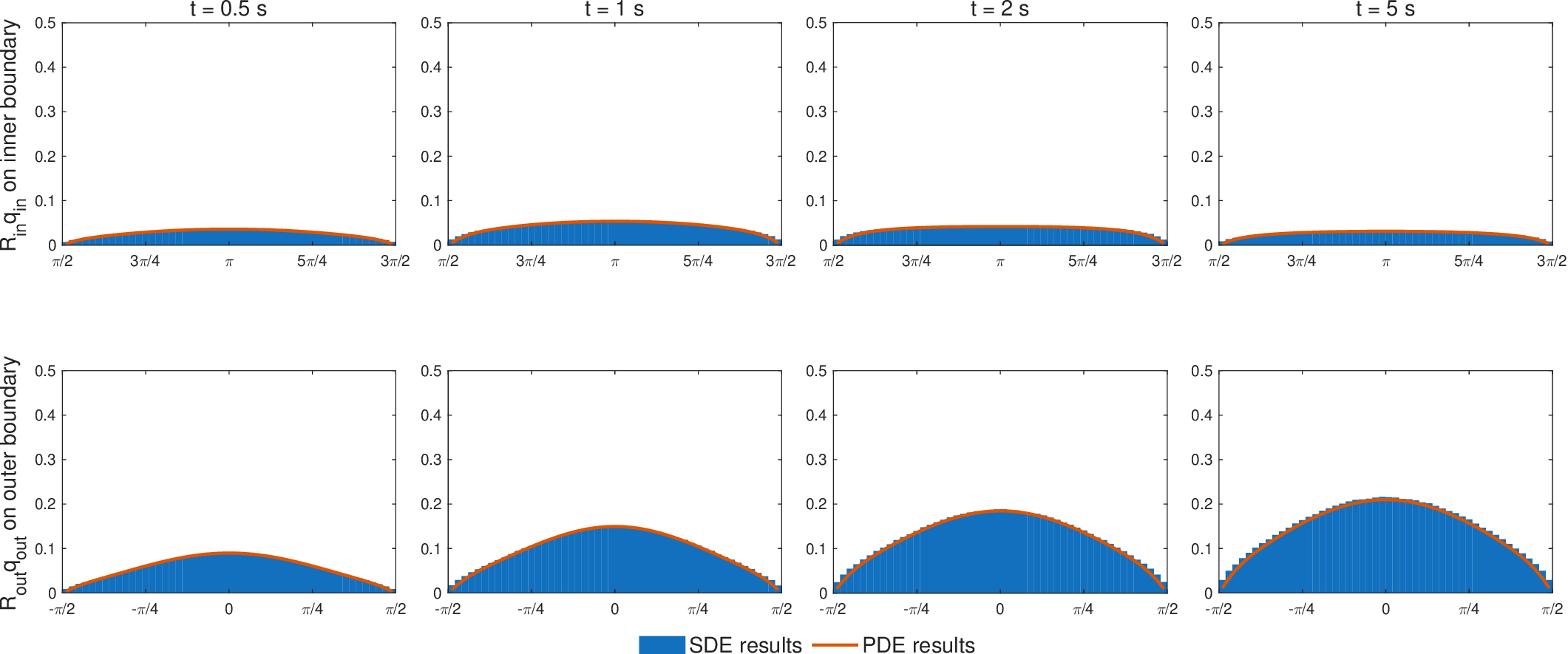}
    \caption{Comparison between SDE simulation results (bar) and PDE numerical results (curve). From top to bottom: probability density on the inner circle, probability density on the outer circle. From left to right: $t=0.5s$, $1s$, $2s$, $5s$.}\label{fig:annulusq} 
\end{figure}

\subsection{Boundary Accumulation with Different Parameters}\label{ch:5.2.2}

\paragraph{Boundary accumulation in a circular Domain} 
Figure \ref{fig:circularRvariance} illustrates how the size of the domain and the strength of the rotational diffusion process would affect the probability of particles that aggregate at the boundary. Here we plot the  probability as a function of radius $R$ for different $D_\theta$. 

 For self-propelled particles confined within a disk, at steady-state, the fraction of particles at the boundary increases as the disk radius $R$ decreases. Physically, particles in a smaller disk encounter the boundary more frequently and remain there until their swimming direction allows them to return to the interior. This trend is qualitatively consistent with the experimental observations and simulations in \cite{ostapenko2018curvature}, where the authors observe stronger microalgal accumulation near confining walls of higher curvature. At a fixed radius, increasing the rotational diffusion coefficient $D_\theta$ reduces boundary accumulation: stronger angular fluctuations help particles reorient toward the interior, shortening their residence time at the boundary.

\begin{figure}[!tbh]
    \centering
    \includegraphics[scale=0.5]{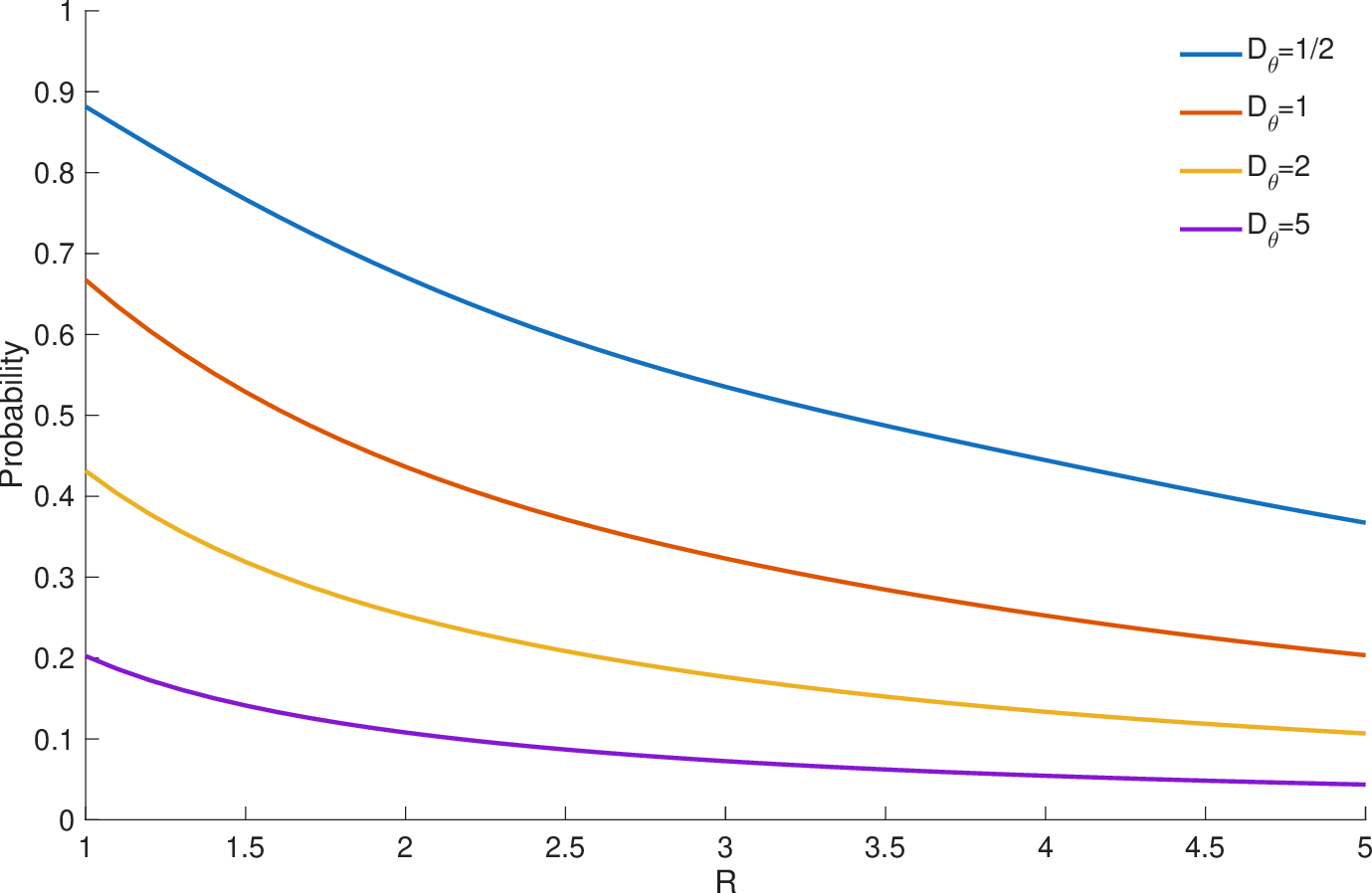}
    \caption{Steady-state probability of finding a particle at the boundary of a circular domain as a function of the domain radius \(R\). Different curves correspond to different rotational diffusion coefficients \(D_\theta\). The time step is \(\Delta t = 0.001\).}
    \label{fig:circularRvariance}
\end{figure}

To further investigate the distributions of particles at steady states, we consider the marginal probability density function
\begin{equation}
    \Phi(r) =\int_{-\pi/2}^{3\pi/2} rp(r,\eta)\dd\eta\cdot\chi(0<r<R)+\int_{-\pi/2}^{\pi/2} Rq(\eta)\dd\eta\cdot\chi(r=R),
\end{equation}
and the results are shown in Figure \ref{fig:circulardensity}. Here for $0<r<R$, $\Phi(r)$ denotes the radial probability density, whereas $\Phi(R)$ denotes the probability density at the boundary. The peaks all occur near the boundaries of the domains under consideration, indicating the presence of boundary accumulation. The marginal density profiles become nearly linear away from the boundaries, implying uniform distribution in radial direction in the interior of the domains. Figure \ref{fig:circularprobability} shows the probability of finding particles within a radial interval of width $\Delta r$ centered at $r$, obtained by integrating the $\Phi(r)$. The pdf increases sharply near the boundary while remaining linear inside the domains, in agreement with the boundary accumulation phenomenon already observed in Figure \ref{fig:circulardensity}.
\begin{figure}[!tbh]
    \centering
    \includegraphics[scale=0.5]{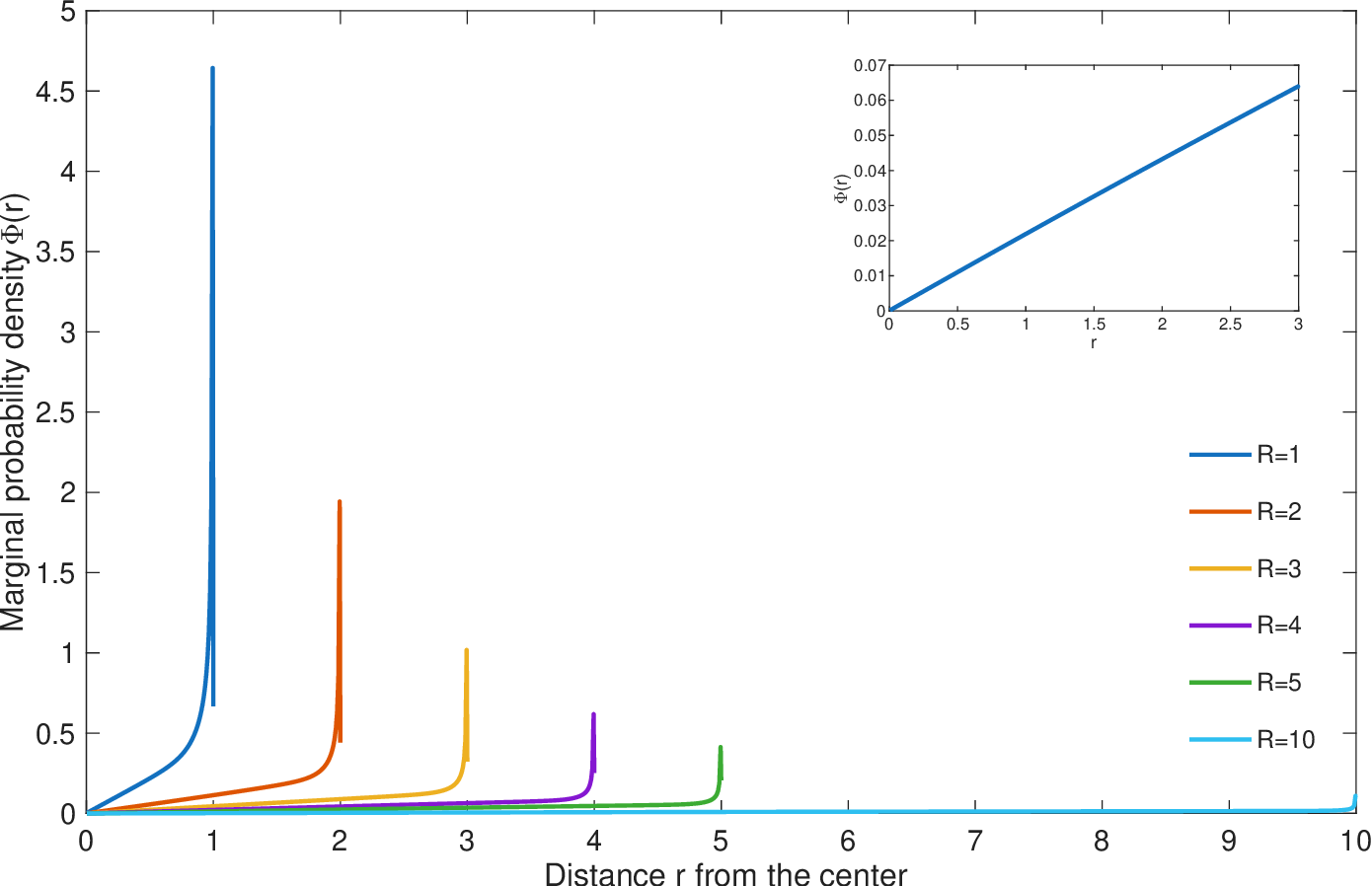}
    \caption{Marginal probability density $\Phi(r)$ in domains with different radii $R$ at their respective steady states. (Inset: zoomed-in view of the results when $R=4$ in the bulk region.) $V = 1$, $D_\theta = 1$ and $\Delta t=0.001$ for each curve.}
    \label{fig:circulardensity}
\end{figure}
\begin{figure}[!tbh]
    \centering
    \includegraphics[scale=0.5]{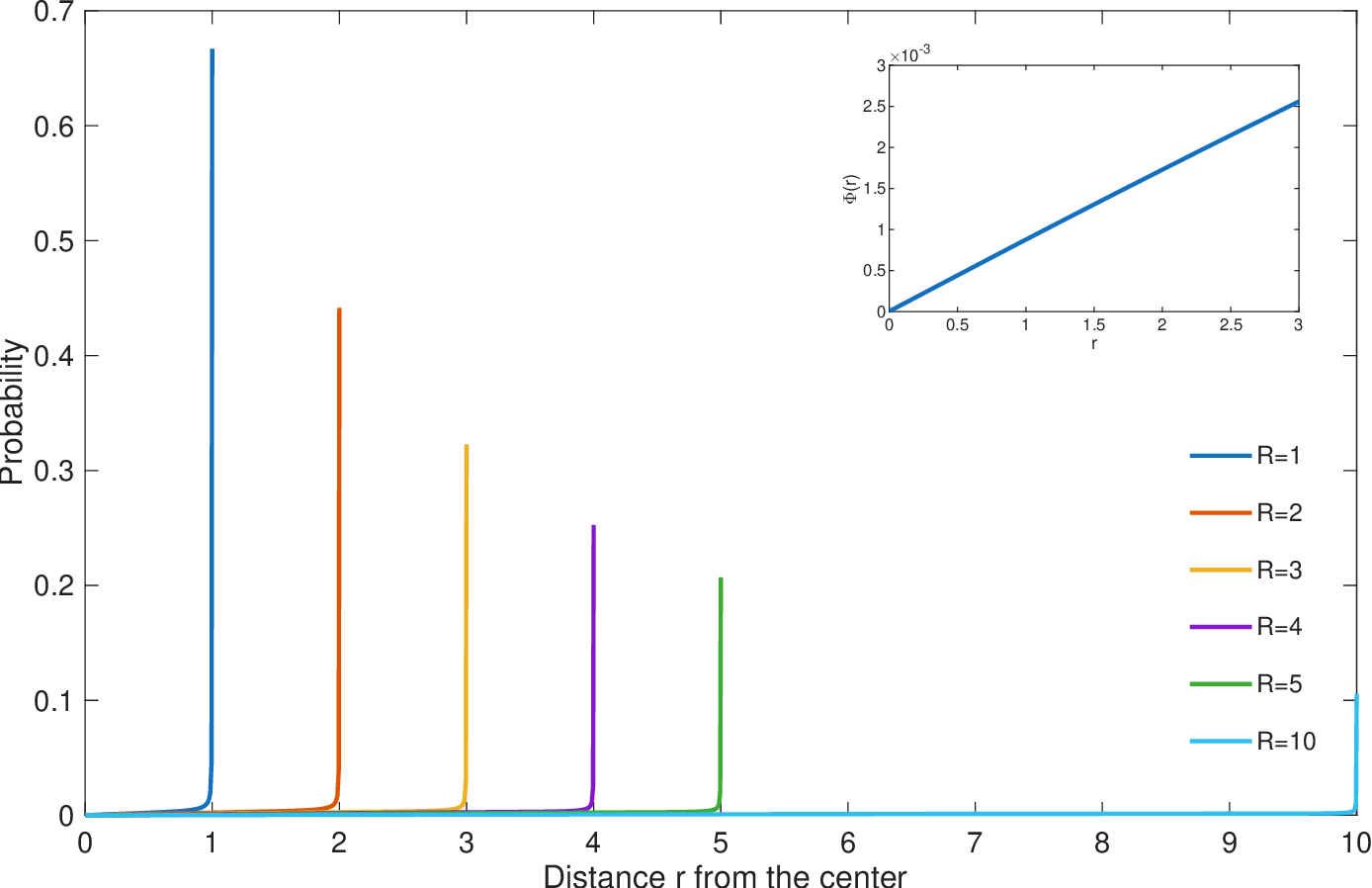}
    \caption{Probability distribution in domains with different radii $R$ at their respective steady states. (Inset: zoomed-in view of the results when $R=4$ in the bulk region.) $V = 1$, $D_\theta = 1$ and $\Delta t=0.001$ for each curve.}
    \label{fig:circularprobability}
\end{figure}

\paragraph{Boundary accumulation in an annular domain.}
To demonstrate the effects of shape parameters $R_\mathrm{in}$, $R_\mathrm{out}$ and diffusion strength $D_\theta$, we consider $P = \iint rp(t,r,\eta)\dd\eta\dd r$, $Q_\mathrm{in} = \int R_\mathrm{in}q_\mathrm{in}(\eta)\dd\eta$ and $Q_\mathrm{out} = \int R_\mathrm{out}q_\mathrm{out}(\eta)\dd\eta$ to represent the probabilities of finding particles inside the domain, moving along the inner boundary and the outer boundary respectively.

Figure~\ref{fig:annulusRinvariance} shows the steady-state probabilities of finding a particle in the interior, $P$, at the inner boundary, $Q_\mathrm{in}$, and at the outer boundary, $Q_\mathrm{out}$, as functions of the inner radius $R_\mathrm{in}$ for different rotational diffusion coefficients $D_\theta$, with $R_\mathrm{out}=5$. As $R_\mathrm{in}$ increases, the annulus becomes narrower: $P$ decreases, whereas both boundary probabilities increase. This trend is consistent with more frequent encounters with the boundaries as the available interior space shrinks. Throughout the range considered, $Q_\mathrm{out}>Q_\mathrm{in}$, indicating that a larger fraction of particles resides at the outer boundary. Increasing $D_\theta$ reduces boundary accumulation and increases the probability of finding particles in the interior, consistent with the results for circular domains in Section~\ref{ch:5.2.2}. Stronger angular fluctuations facilitate reorientation toward the interior and thus promote escape from the boundaries.

\begin{figure}[!tbh]
\centering
\begin{minipage}{0.47\textwidth}
\centering
\includegraphics[height=3.5cm]{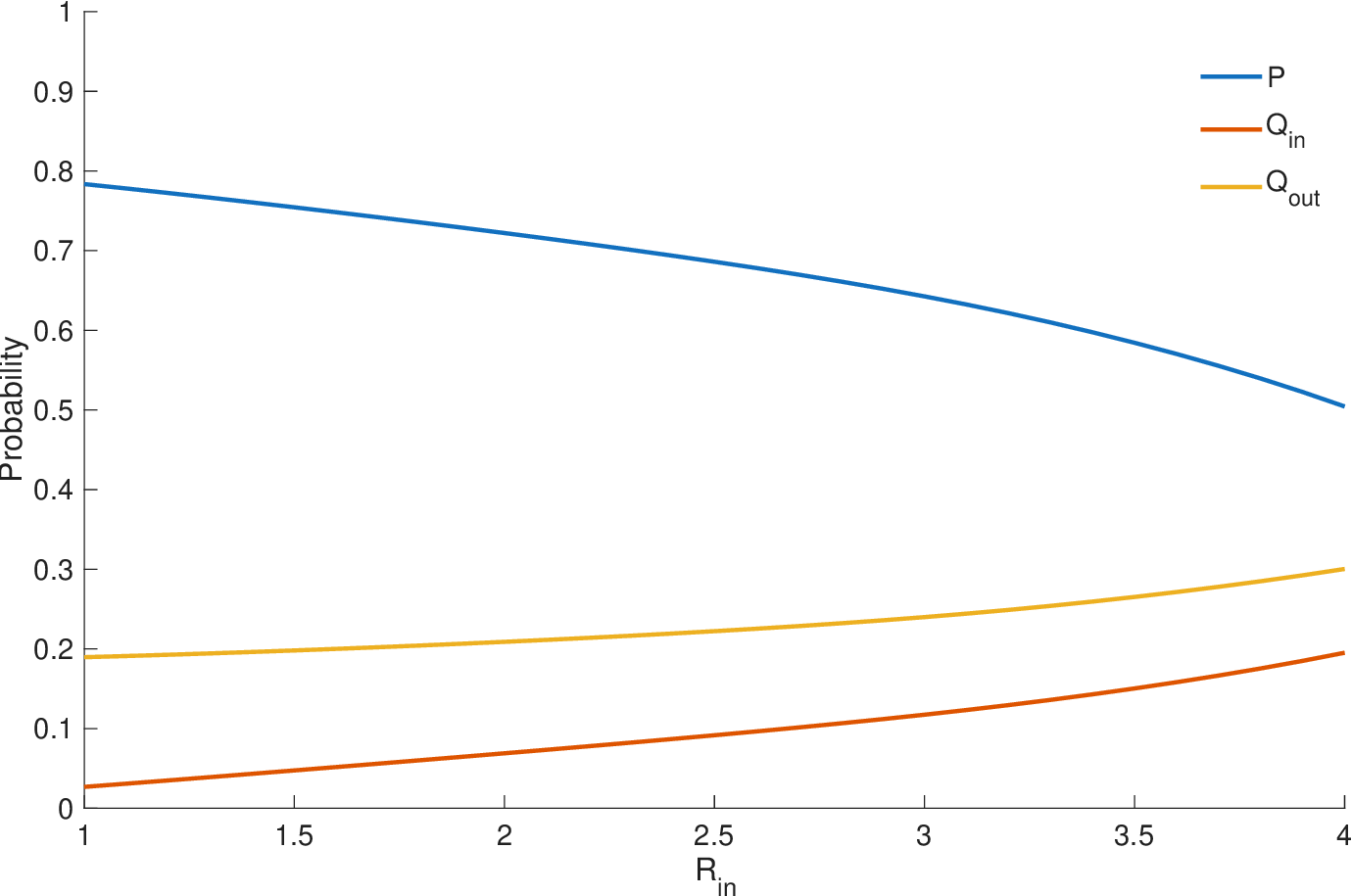}
\end{minipage}
\begin{minipage}{0.47\textwidth}
\centering
\includegraphics[height=3.5cm]{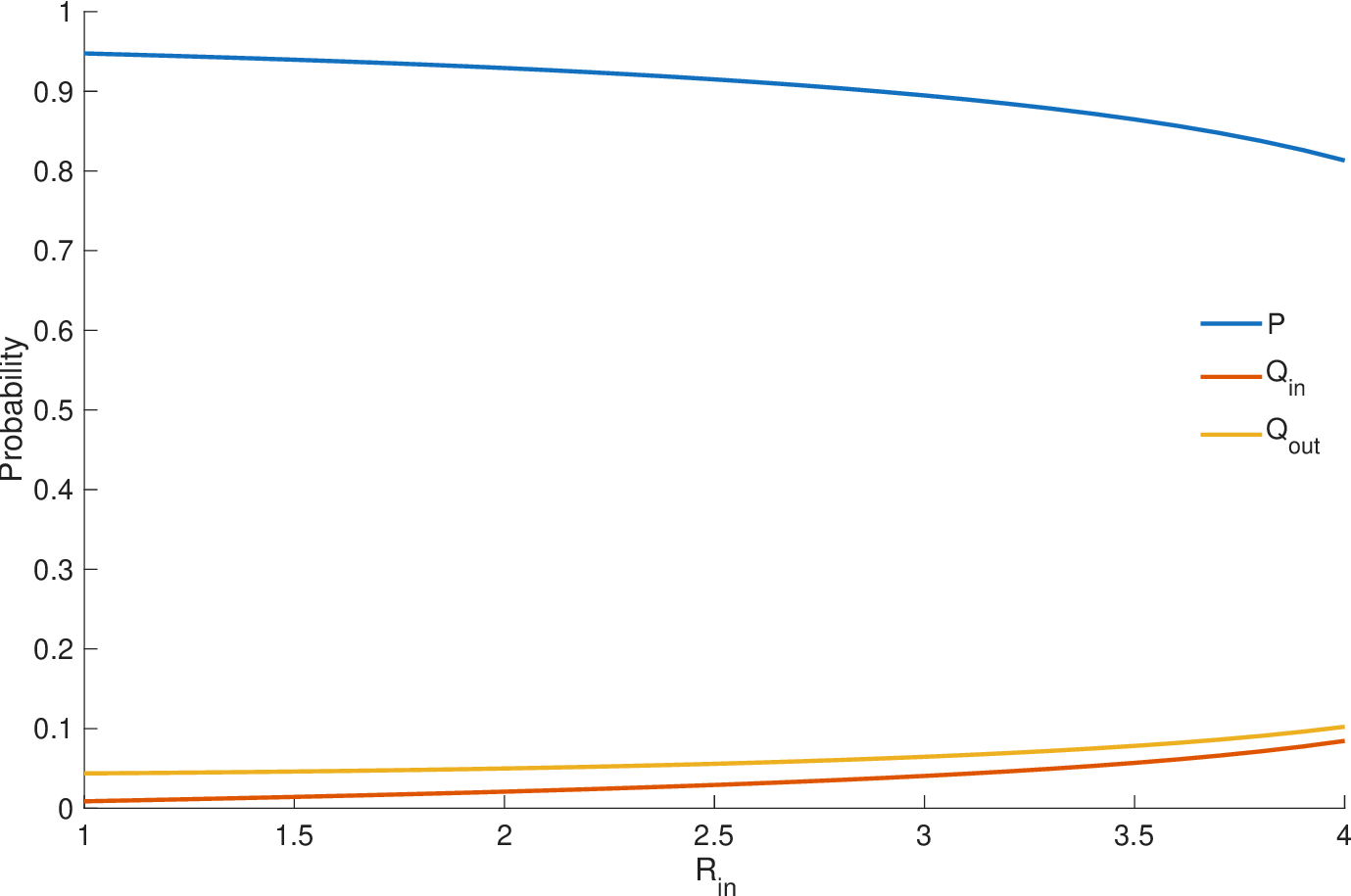}
\end{minipage}
\caption{Steady-state probabilities of finding a particle in the interior, $P=\iint r p(t,r,\eta)\,\dd\eta\,\dd r$, at the inner boundary, $Q_\mathrm{in}=\int R_\mathrm{in}q_\mathrm{in}(\eta)\,\dd\eta$, and at the outer boundary, $Q_\mathrm{out}=\int R_\mathrm{out}q_\mathrm{out}(\eta)\,\dd\eta$, as functions of the inner radius $R_\mathrm{in}$ of an annular domain. The swimming speed and outer radius are fixed at $V=1$ and $R_\mathrm{out}=5$. Left: $D_\theta=1$; right: $D_\theta=5$.}
\label{fig:annulusRinvariance}
\end{figure}

Figure~\ref{fig:annulusRoutvariance} shows the steady-state probabilities $P$, $Q_\mathrm{in}$, and $Q_\mathrm{out}$ as functions of the outer radius $R_\mathrm{out}$ for different rotational diffusion coefficients $D_\theta$, with $R_\mathrm{in}=1$. Increasing $R_\mathrm{out}$ widens the annulus and provides more interior space, increasing the probability $P$ of finding particles in the interior. For the range considered, $Q_\mathrm{out}>Q_\mathrm{in}$, so a larger fraction of particles resides at the outer boundary than at the inner boundary. Increasing $D_\theta$ further promotes particle residence in the interior and reduces boundary accumulation, consistent with the preceding results: stronger angular fluctuations facilitate reorientation away from the boundaries.
\begin{figure}[!tbh]
\centering
\begin{minipage}{0.47\textwidth}
\centering
\includegraphics[height=3.5cm]{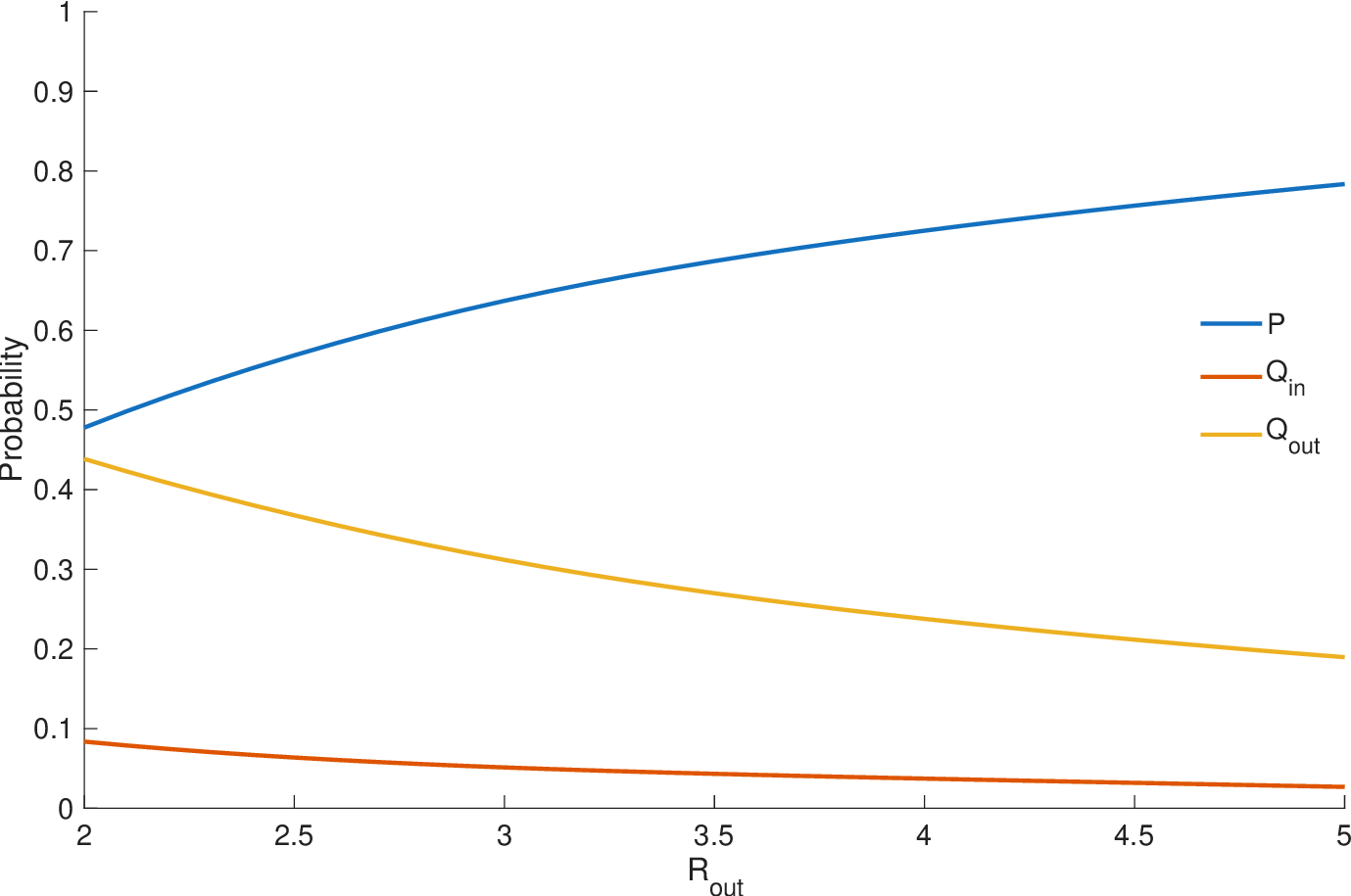}
\end{minipage}
\begin{minipage}{0.47\textwidth}
\centering
\includegraphics[height=3.5cm]{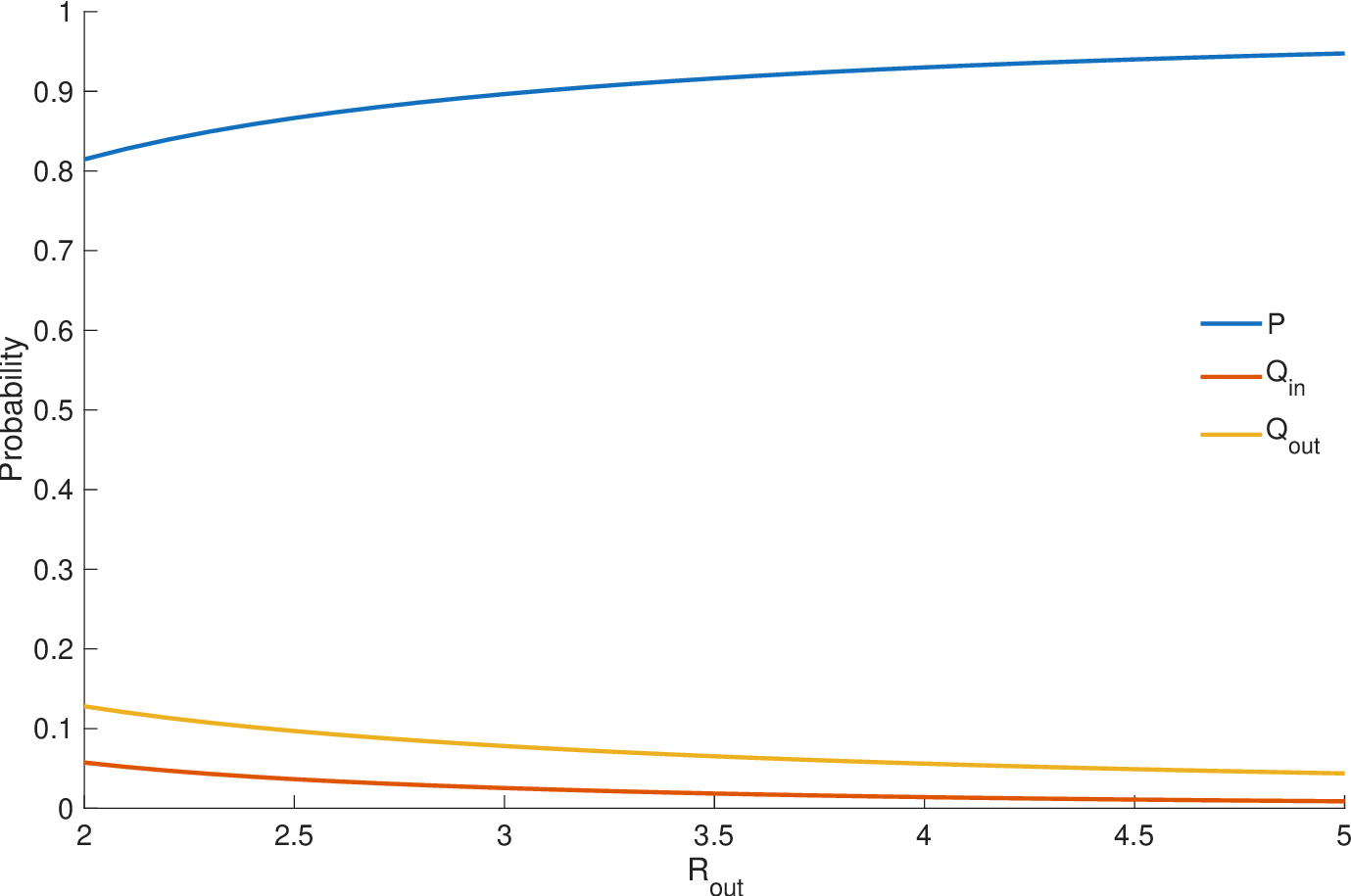}
\end{minipage}
\caption{Steady-state probabilities of finding a particle in the interior, $P=\iint r p(t,r,\eta)\,\dd\eta\,\dd r$, at the inner boundary, $Q_\mathrm{in}=\int R_\mathrm{in}q_\mathrm{in}(\eta)\,\dd\eta$, and at the outer boundary, $Q_\mathrm{out}=\int R_\mathrm{out}q_\mathrm{out}(\eta)\,\dd\eta$, as functions of the outer radius $R_\mathrm{out}$ of an annular domain. The swimming speed and inner radius are fixed at $V=1$ and $R_\mathrm{in}=1$. Left: $D_\theta=1$; right: $D_\theta=5$.}
\label{fig:annulusRoutvariance}
\end{figure}

\section{Conclusion}\label{ch:6}
In this paper, individual-based models and Fokker-Planck models are developed to investigate the boundary accumulation dynamics of self-propelled particles in confined two-dimensional circular and annular domains. The proposed framework explicitly incorporates the exchange between the interior free-swimming state and the boundary-contact state through suitable boundary conditions, allowing the boundary accumulation phenomenon to be characterized within a unified probabilistic formulation. Numerical experiments and Monte-Carlo simulations are carried out and compared to validate the consistency between the continuum and stochastic descriptions. The influence of geometric and dynamical parameters on the distributions is also investigated and discussed.

Beyond the results presented in this paper, our models provide a basis for further mathematical analysis and a deeper theoretical understanding of the mechanisms underlying boundary accumulation. Additionally, we intend to extend the framework to confined domains with more complex geometries, such as ellipses, where the behavior of self-propelled particles remains an active area as studied in \cite{cammann2021emergent,ostapenko2018curvature}.

\section*{Acknowledgments}

At the end of paper but preceding the References.

The author would like to thank  ....


\end{document}